\documentclass{amsart}

 \usepackage{verbatim}
\usepackage{mathrsfs}
\usepackage{latexsym}
\usepackage{epsfig}
\usepackage{amsmath}
\usepackage{bbm}
\usepackage{graphics}
\usepackage{amssymb}
\usepackage[alphabetic]{amsrefs}
\usepackage{amsopn}
\usepackage{amscd}
\usepackage[all,knot]{xy}
\xyoption{all}
\usepackage{rotating}
\usepackage{hyperref}

\numberwithin{equation}{section}

\theoremstyle{plain}
\newtheorem{thm}{Theorem}[section]
\newtheorem{lem}[thm]{Lemma}
\newtheorem{prop}[thm]{Proposition}

\newtheorem*{thm*}{Theorem}
\newtheorem*{lem*}{Lemma}
\newtheorem*{prop*}{Proposition}
\newtheorem*{cor*}{Corollary}

\theoremstyle{definition}
\newtheorem{defn}[thm]{Definition}
\newtheorem*{defn*}{Definition}

\newtheorem{ex}[thm]{Example}

{}
\newtheorem{rem}[thm]{Remark}
\newtheorem*{rem*}{Remark}

\newtheorem{notation}[thm]{Notation}{}
{}
{}
{}
\newtheorem{aside}[thm]{Aside}

\theoremstyle{remark}
{}
{}
{}

\def\cie{\subseteq}

\def\iso{\cong}

\def\Un{\bigcup}
\def\un{\cup}

\def\intersec{\cap}

\def\to{\longrightarrow}

\def\rimp{\Rightarrow}

\def\ZZ{\mathbb{Z}}

\def\a{\alpha}

\def\e{\epsilon}
\def\l{\lambda}
\def\r{\rho}
\def\s{\sigma}
\def\S{\Sigma}
\def\t{\tau}
\def\D{\mathsf{D}}
\def\K{\mathsf{K}}
\def\sfT{\mathsf{T}}
\def\sfS{\mathsf{S}}
\def\sfI{\mathsf{I}}
\def\sfP{\mathsf{P}}
\def\sfM{\mathsf{M}}
\def\sfQ{\mathsf{Q}}
\def\sfL{\mathsf{L}}

\def\sfC{\mathsf{C}}
\def\sfN{\mathsf{N}}

\DeclareMathOperator{\Spec}{Spec}

\DeclareMathOperator{\Sing}{Sing}
\DeclareMathOperator{\Proj}{Proj}
\DeclareMathOperator{\Spc}{Spc}

\DeclareMathOperator{\supp}{supp}

\DeclareMathOperator{\Hom}{Hom}

\DeclareMathOperator{\cone}{cone}

\DeclareMathOperator{\hocolim}{hocolim}
\DeclareMathOperator{\modu}{mod}
\DeclareMathOperator{\Modu}{Mod}
\DeclareMathOperator{\smodu}{\underline{\modu}}
\DeclareMathOperator{\sModu}{\underline{\Modu}}

\DeclareMathOperator{\Flat}{Flat}

\DeclareMathOperator{\Inj}{Inj}

\DeclareMathOperator{\thick}{thick}
\DeclareMathOperator{\Thick}{Thick}
\DeclareMathOperator{\loc}{loc}
\DeclareMathOperator{\Loc}{Loc}
\DeclareMathOperator{\Thom}{Thom}

\title[Support theory for triangulated categories]{A tour of support theory for triangulated categories through tensor triangular geometry}
\author{Greg Stevenson}
\address{Greg Stevenson, Universit\"at Bielefeld, Fakult\"at f\"ur Mathematik, BIREP Gruppe, Postfach 10\,01\,31, 33501 Bielefeld, Germany.}
\email{gstevens@math.uni-bielefeld.de}

\thanks{I am grateful to the CRM for both the very pleasant and productive environment provided and the opportunity to give the course which led to these notes during the IRTATCA program. Thanks are also due to the participants in the course for helpful feedback and pointing out several typographical misfortunes in the original version of these notes.}

\begin{document}


\begin{abstract}
\noindent These notes attempt to give a short survey of the approach to support theory and the study of lattices of triangulated subcategories through the machinery of tensor triangular geometry. One main aim is to introduce the material necessary to state and prove the local-to-global principle. In particular, we discuss Balmer's construction of the spectrum, generalised Rickard idempotents and support for compactly generated triangulated categories, and actions of tensor triangulated categories. Several examples are also given along the way. These notes are based on a series of lectures given during the Spring 2015 program on `Interactions between Representation Theory, Algebraic Topology and Commutative Algebra' (IRTATCA) at the CRM in Barcelona.
\end{abstract}

\maketitle

\tableofcontents

\section*{Introduction}

The aim of these notes is to give both an introduction to, and an overview of, certain aspects of the developing field of tensor triangular geometry and abstract support varieties. We seek to understand the coarse structure, i.e.\ lattices of suitable subcategories, of triangulated categories. Put another way, given a triangulated category $\sfT$ and objects $X,Y\in \sfT$, we study the question of when one can obtain $Y$ from $X$ by taking cones, suspensions, and (possibly infinite) coproducts.

The notes are structured as follows. In the first section we review work of Paul Balmer in the case of essentially small (rigid) tensor triangulated categories. Here one can introduce a certain topological space, the spectrum, which solves the classification problem for thick tensor ideals. The spectrum is the universal example of a so-called ``support variety'' and provides a conceptual framework which unites earlier results in various examples such as perfect complexes over schemes, stable categories of modular representations, and the finite stable homotopy category. In this section we first meet the notion of supports and lay the foundations for defining supports in more general settings.

In the second section we turn to the infinite case, i.e.\ compactly generated tensor triangulated categories. Following work of Balmer and Favi we use the compact objects and the Balmer spectrum to define a notion of support for objects of such categories. Along the way we discuss smashing localisations and the associated generalised Rickard idempotents which are the key to the definition of the support given in \cite{BaRickard}.

The third section serves as an introduction to actions of tensor triangulated categories. This allows us to define a relative version of the supports introduced in the second section; in this way we can, at least somewhat, escape the tyranny of monoidal structures. After introducing actions, the associated support theory, and outlining some of the fundamental lemmas concerning supports and actions, we come to the main abstract result of this course - the local-to-global principle. This theorem, which already provides new insight in the situation of Section 2, reduces the study of lattices of localising subcategories to the computation of these lattices in smaller (and hopefully simpler) subcategories.

Finally, the fourth section focusses on illustrating some applications of the abstract machinery from the preceding sections; there will be, of course, examples along the way, but in this section we concentrate on giving more details on some more recent examples where the machinery of actions has been successfully applied. In particular, we discuss singularity categories of affine hypersurfaces and the corresponding classification problem for localising subcategories as well as applications to studying derived categories of representations of quivers over arbitrary commutative noetherian rings.

\section{The Balmer spectrum}

The main reference for this section is the paper \cite{BaSpec} by Paul Balmer. We follow his exposition fairly closely, albeit with two major differences: firstly, in order to simplify the discussion, and since we will not require it later on, we do not work in full generality, and secondly we omit many of the technical details. The interested reader should consult \cite{BaSpec}, \cite{BaSSS} and the references within for further details.

This section mainly consists of definitions, but we provide several examples on the way and ultimately get to Balmer's classification of thick tensor ideals in terms of the spectrum.

\subsection{Rigid tensor triangulated categories}

Throughout this section $\K$ will denote an essentially small triangulated category. We use lowercase letters $k,l,m$ for objects of $\K$ and denote its suspension functor by $\S$.

We begin by introducing the main player in this section.

\begin{defn}
An \emph{essentially small tensor triangulated category} is a triple $(\K, \otimes, \mathbf{1})$, where $\K$ is an essentially small triangulated category and $(\otimes, \mathbf{1})$ is a symmetric monoidal structure on $\K$ such that $\otimes$ is an exact functor in each variable. Slightly more explicitly,
\begin{displaymath}
-\otimes-\colon \K\times \K \to \K
\end{displaymath}
is a symmetric monoidal structure on $\K$ with unit $\mathbf{1}$ and with the property that, for all $k\in \K$, the endofunctors $k\otimes -$ and $-\otimes k$ are exact.
\end{defn}

\begin{rem}
Throughout we shall not generally make explicit the associativity, symmetry, and unit constraints for the symmetric monoidal structure on a tensor triangulated category. By standard coherence results for monoidal structures this will not get us into any trouble.
\end{rem}

\begin{defn}\label{defn_rigid}
Let $\K$ be an essentially small tensor triangulated category. Assume that $\K$ is closed symmetric monoidal, i.e.\ for each $k\in \K$ the functor $k\otimes-$ has a right adjoint which we denote $\hom(k,-)$. These functors can be assembled into a bifunctor $\hom(-,-)$ which we call the \emph{internal hom} of $\K$. By definition one has, for all $k,l,m \in \K$, the tensor-hom adjunction
\begin{displaymath}
\K(k\otimes l, m) \cong \K(l, \hom(k,m)),
\end{displaymath}
with corresponding units and counits
\begin{displaymath}
\eta_{k,l} \colon l \to \hom(k,k\otimes l) \quad \text{and} \quad \epsilon_{k,l} \colon \hom(k,l)\otimes k \to l.
\end{displaymath}
The \emph{dual} of $k\in \K$ is the object
\begin{displaymath}
k^\vee = \hom(k,\mathbf{1}).
\end{displaymath}
Given $k,l \in \K$ there is a natural \emph{evaluation map}
\begin{displaymath}
k^\vee \otimes l \to \hom(k,l),
\end{displaymath}
which is defined by following the identity map on $l$ through the composite
\begin{displaymath}
\xymatrix{
\K(l,l) \ar[r]^-\sim & \K(l\otimes \mathbf{1}, l) \ar[rr]^-{\K(l\otimes \e_{k,\mathbf{1}}, l)} && \K(k\otimes k^\vee \otimes l, l) \ar[r]^-\sim & \K(k^\vee \otimes l, \hom(k,l)).
}
\end{displaymath}
We say that $\K$ is \emph{rigid} if for all $k,l \in \K$ this natural evaluation map is an isomorphism
\begin{displaymath}
k^\vee \otimes l \stackrel{\sim}{\to} \hom(k,l).
\end{displaymath}
\end{defn}

\begin{rem}
If $\K$ is rigid then, given $k\in \K$, there is a natural isomorphism
\begin{displaymath}
(k^\vee)^\vee \cong k,
\end{displaymath}
and the functor $k^\vee\otimes -$ is both a left and a right adjoint to $k\otimes-$, the functor given by tensoring with $k$.
\end{rem}

\begin{ex}
Let us provide some standard examples of essentially small rigid tensor triangulated categories:
\begin{itemize}
\item[(1)] Given a commutative ring $R$ the category $\D^\mathrm{perf}(R)$ of \emph{perfect complexes} i.e., those complexes in the derived category which are quasi-isomorphic to a bounded complex of finitely generated projectives, is symmetric monoidal via the left derived tensor product $\otimes^\mathbf{L}_R$ (which we will usually denote just as $\otimes_R$ or $\otimes$ if the ring is clear) with unit the stalk complex $R$ sitting in degree $0$. The category $\D^\mathrm{perf}(R)$ is easily checked to be rigid.
\item[(2)] Let $G$ be a finite group and $k$ be a field whose characteristic divides the order of $G$ (this is not necessary but rules out trivial cases). We write $\modu kG$ for the category of finite dimensional $kG$-modules. This is a Frobenius category, so its stable category $\smodu kG$, which is obtained by factoring out maps factoring through projectives, is triangulated (see \cite{Happeltricat} for instance for more details). Moreover, it is a rigid tensor triangulated category via the usual tensor product $\otimes_k$ with the diagonal action and unit object the trivial representation $k$.
\item[(3)] The finite stable homotopy category $\mathrm{SH}^\mathrm{fin}$ together with the smash product of spectra is a rigid tensor triangulated category with unit object the sphere spectrum $S^0$.
\end{itemize}
\end{ex}

\subsection{Ideals and the spectrum}
From this point onward $\K$ denotes an essentially small rigid tensor triangulated category with tensor product $\otimes$ and unit $\mathbf{1}$. For much of what follows the rigidity assumption is overkill, but it simplifies the discussion in several places and will be a necessary hypotheses in the sections to come.

We begin by recalling the subcategories of $\K$ with which we will be concerned. All subcategories of $\K$ are assumed to be full and replete (i.e. closed under isomorphisms). The simple, but beautiful, idea is to view $\K$ as a very strange sort of ring. One takes thick subcategories to be the analogue of additive subgroups and then defines ideals and prime ideals in the naive way.

\begin{defn}
A triangulated subcategory $\sfI$ of $\K$ is \emph{thick} if it is closed under taking direct summands, i.e.\ if $k\oplus k'\in \sfI$ then both $k$ and $k'$ lie in $\sfI$. To be very explicit: a full subcategory $\sfI$ of $\K$ is thick if it is closed under suspensions, cones, and direct summands. 

A thick subcategory $\sfI$ of $\K$ is a (thick) \emph{tensor-ideal} if given any $k\in \K$ and $l\in \sfI$ the tensor product $k\otimes l$ lies in $\sfI$. Put another way we require the functor
\begin{displaymath}
\K \times \sfI \stackrel{\otimes}{\to} \K
\end{displaymath}
to factor through $\sfI$.

Finally, a proper thick tensor-ideal $\sfP$ of $\K$ is \emph{prime} if given $k,l\in \K$ such that $k\otimes l \in \sfP$ then at least one of $k$ or $l$ lies in $\sfP$. We will often just call prime tensor-ideals in $\K$ prime ideals.
\end{defn}

\begin{rem}
Since we have assumed $\K$ is essentially small the collections of thick subcategories, thick tensor-ideals, and prime tensor-ideals each form a set.
\end{rem}

\begin{rem}\label{rem_rigid_facts}
Rigidity of $\K$ provides the following simplification when dealing with tensor-ideals. Given $k\in \K$ the adjunctions between $k\otimes-$ and $k^\vee\otimes-$ imply that $k$ is a summand of $k\otimes k \otimes k^\vee$ and that $k^\vee$ is a summand of $k\otimes k^\vee \otimes k^\vee$. The former implies that every tensor-ideal is radical, i.e.\ if $\sfI$ is a tensor-ideal and $k^{\otimes n} \in \sfI$ then $k\in \sfI$. The latter implies that every tensor-ideal is closed under taking duals.

An important special case of the above discussion is the following: if $k\in \K$ is nilpotent, i.e.\ $k^{\otimes n}\cong 0$, then $k\cong 0$.
\end{rem}

Given a collection of objects $S\cie \K$ we denote by $\thick(S)$ (resp.\ $\thick^\otimes(S)$) the smallest thick subcategory (resp.\ thick tensor-ideal) containing $S$. We use $\Thick(\K)$ and $\Thick^\otimes(K)$ respectively to denote the sets of thick subcategories and thick tensor-ideals of $\K$. Both of these sets of subcategories are naturally ordered by inclusion and form complete lattices whose meet is given by intersection.

\begin{lem}\label{lem_small_unit_gen}
Suppose that $\K$ is generated by the tensor unit, i.e.\ $\thick(\mathbf{1}) = \K$. Then every thick subcategory is a tensor-ideal:
\begin{displaymath}
\Thick(\K) = \Thick^\otimes(K).
\end{displaymath}
\end{lem}
\begin{proof}
Exercise.
\end{proof}

\begin{ex}
The lemma applies to both $\D^\mathrm{perf}(R)$ and $\mathrm{SH}^\mathrm{fin}$ which are generated by their respective tensor units $R$ and $S^0$. It is not necessarily the case that $\smodu kG$ is generated by its tensor unit $k$. However, if $G$ is a $p$-group then the trivial module $k$ is the unique simple  $kG$-module and thus generates $\smodu kG$.
\end{ex}

The prime ideals, somewhat unsurprisingly, receive special attention (and notation).

\begin{defn}
The \emph{spectrum of} $\K$ is the set
\begin{displaymath}
\Spc\K = \{\sfP \cie \K \; \vert \; \sfP \;\text{is prime}\}
\end{displaymath}
of prime ideals of $\K$. 
\end{defn}

We next record some elementary but crucial facts about prime ideals and $\Spc \K$.

\begin{prop}[\cite{BaSpec}*{Proposition~2.3}]
Let $\K$ be as above.
\begin{itemize}
\item[$(a)$] Let $S$ be a set of objects of $\K$ containing $\mathbf{1}$ and such that if $k,l \in S$ then $k\otimes l \in S$. If $S$ does not contain $0$ then there exists a $\sfP\in \Spc \K$ such that $\sfP \cap S = \varnothing$.
\item[$(b)$] For any proper thick tensor-ideal $\sfI \subsetneq \K$ there exists a maximal proper thick tensor-ideal $\sfM$ with $\sfI \cie \sfM$.
\item[$(c)$] Maximal proper thick tensor-ideals are prime.
\item[$(d)$] The spectrum is not empty: $\Spc \K \neq \varnothing$. 
\end{itemize}
\end{prop}

\begin{rem}
It is worth noting that, as in the common proof of the analogous statement in commutative algebra, the proof of this proposition appeals to Zorn's lemma.
\end{rem}

\subsection{Supports and the Zariski topology}

We now define, for each object $k\in \K$, a subset of prime ideals at which $k$ is ``non-zero''. This is the central construction of the section, allowing us to both put a topology on $\Spc \K$ and to understand $\Thick^\otimes(\K)$ in terms of the resulting topological space.

\begin{defn}
Let $k$ be an object of $\K$. The \emph{support} of $k$ is the subset
\begin{displaymath}
\supp k = \{\sfP \in \Spc\K \; \vert \; k\notin \sfP\}.
\end{displaymath}
We denote by
\begin{displaymath}
U(k) = \{\sfP\in \Spc \K \; \vert \; k\in \sfP\}
\end{displaymath}
the complement of $\supp k$. 
\end{defn}

Let us give some initial intuition for the way in which one can think of $k$ as being supported at $\sfP\in \supp k$. We can form the Verdier quotient $\K/\sfP$ of $\K$ by $\sfP$, which comes with a canonical projection $\pi\colon \K \to \K/\sfP$. Since $\sfP$ is thick the kernel of $\pi$ is precisely $\sfP$ and so, as $k\notin \sfP$, the object $\pi(k)$ is non-zero in the quotient; this is the sense in which $k$ is supported at $\sfP$. The reason this may, at first, look at odds with the analogous definition in commutative algebra is that one inverts morphisms in a triangulated category by taking such quotients. One can, at least in many cases, make a precise connection between the support defined above and the usual support of modules over a (graded) commutative ring.

We now list the main properties of the support.

\begin{lem}[\cite{BaSpec}*{Lemma~2.6}]\label{lem_supp_props}
The assignment $k \mapsto \supp k$ given by the support satisfies the following properties:
\begin{itemize}
\item[$(a)$] $\supp \mathbf{1} = \Spc\K$ and $\supp 0 = \varnothing$;
\item[$(b)$] $\supp(k\oplus l) = \supp(k) \cup \supp(l)$;
\item[$(c)$] $\supp(\S k) = \supp k$;
\item[$(d)$] for any distinguished triangle
\begin{displaymath}
k \to l \to m \to \S k
\end{displaymath}
in $\K$ there is a containment 
\begin{displaymath}
\supp l \cie (\supp k \cup \supp m);
\end{displaymath}
\item[$(e)$] $\supp(k\otimes l) = \supp(k) \cap \supp(l)$.
\end{itemize}
\end{lem}
\begin{proof}
We refer to Balmer's paper for the details, where the corresponding results are proved for the subsets $U(k)$. These properties, as stated above, appear in \cite{BaSpec}*{Definition~3.1}. However, we suggest proving these statements as doing so provides an instructive exercise.
\end{proof}

Another very important property of the support is that it can detect whether or not an object is zero.

\begin{lem}
Given $k\in \K$ we have $\supp k = \varnothing$ if and only if $k\cong 0$.
\end{lem}
\begin{proof}
By \cite{BaSpec}*{Corollary~2.4} the support of $k$ is empty if and only if $k$ is tensor-nilpotent i.e., $k^{\otimes n}\cong 0$. As $\K$ is rigid there are no non-zero tensor-nilpotent objects by Remark~\ref{rem_rigid_facts}.
\end{proof}

A fairly immediate consequence of Lemma~\ref{lem_supp_props} is that the family of subsets $\{\supp k \; \vert \; k\in \K\}$ form a basis of closed subsets for a topology on $\Spc\K$.

\begin{defn}
The \emph{Zariski topology} on $\Spc\K$ is the topology defined by the basis of closed subsets
\begin{displaymath}
\{\supp k \; \vert \; k\in \K\}
\end{displaymath}
The closed subsets of $\Spc\K$ are of the form
\begin{align*}
Z(S) &= \{\sfP \in \Spc\K \; \vert \; S\cap \sfP = \varnothing\} \\
&= \bigcap_{k\in S} \supp(k)
\end{align*}
where $S\cie \K$ is an arbitrary family of objects.
\end{defn}

Before continuing to the universal property of $\Spc\K$ and the classification theorem we discuss some topological properties of $\Spc\K$.

\begin{prop}[\cite{BaSpec}*{Proposition~2.9}]
Let $\sfP \in \Spc\K$. The closure of $\sfP$ is
\begin{displaymath}
\overline{\{\sfP\}} = \{\sfQ\in \Spc\K \; \vert \; \sfQ\cie \sfP\}.
\end{displaymath}
In particular, $\Spc\K$ is $T_0$ i.e., for $\sfP_1,\sfP_2 \in \Spc\K$
\begin{displaymath}
\overline{\{\sfP_1\}} = \overline{\{\sfP_2\}} \; \rimp \; \sfP_1 = \sfP_2.
\end{displaymath}
\end{prop}
\begin{proof}
Let $S_0 = \K\setminus \sfP$ denote the complement of $\sfP$. It is immediate that $\sfP \in Z(S_0)$ and one easily checks that if $\sfP\in Z(S)$ then $S\cie S_0$. Thus for any such $S$ we have $Z(S_0)\cie Z(S)$, i.e.\ $Z(S_0)$ is the smallest closed subset containing $\sfP$. This shows
\begin{displaymath}
\overline{\{\sfP\}} = Z(S_0) = \{\sfQ\in \Spc\K \; \vert \; \sfQ\cie \sfP\}
\end{displaymath}
as claimed. The assertion that $\Spc\K$ is $T_0$ follows immediately.
\end{proof}

\begin{rem}\label{rem_gymnastics}
This proposition is the first indication of the mental gymnastics that occur when dealing with $\Spc \K$ versus $\Spec R$ for a commutative ring $R$. A wealth of further information on the relationship between prime ideals in $R$ and prime tensor-ideals in $\Spc \D^\mathrm{perf}(R)$ can be found in \cite{BaSSS}.
\end{rem}

In fact $\Spc\K$ is much more than just a $T_0$ space.

\begin{defn}\label{defn_spectral}
A topological space $X$ is a \emph{spectral space} if it verifies the following properties:
\begin{itemize}
\item[(1)] $X$ is $T_0$;
\item[(2)] $X$ is quasi-compact;
\item[(3)] the quasi-compact open subsets of $X$ are closed under finite intersections and form an open basis of $X$;
\item[(4)] every non-empty irreducible closed subset of $X$ has a generic point.
\end{itemize}
Given spectral spaces $X$ and $Y$, a \emph{spectral map} $f\colon X\to Y$ is a continuous map such that for any quasi-compact open $U\cie Y$ the preimage $f^{-1}(U)$ is quasi-compact. 
\end{defn}

Typical examples of spectral spaces are given by $\Spec R$ where $R$ is a commutative ring. In fact Hochster has shown in \cite{HochsterSpectral} that any spectral topological space is of this form. Further examples include the topological space underlying any quasi-compact and quasi-separated scheme. Another class of examples is provided by the following result.

\begin{thm}[\cite{BKS}]\label{thm_spectral}
Let $\K$ be as above. The space $\Spc\K$ is spectral.
\end{thm}

Thus the spaces we produce by taking spectra of tensor triangulated categories are particularly nice and enjoy many desirable properties.

\begin{aside}
There is a deep connection between spectral spaces and the theory of particularly nice lattices (coherent frames to be precise). One consequence of this is that the topology of a spectral space is determined by the specialisation ordering on points. This ordering is given by defining, for a spectral space $X$, a point $y\in X$ to be a specialisation of $x\in X$ if $y$ is in the closure of $x$. Viewed through the lens of lattice theory, the key fact which makes Balmer's theory work so well is that the lattice of thick tensor-ideals is distributive. More details on this point of view can be found in \cite{KockPitsch}.
\end{aside}

As a complement to Theorem~\ref{thm_spectral} let us record a few related facts which appear in Balmer's work.

\begin{prop}\label{prop_spc_props}
For any rigid tensor triangulated category $\K$ the following assertions hold.
\begin{itemize}
\item[$(a)$] Any non-empty closed subset of $\Spc\K$ contains at least one closed point.
\item[$(b)$] For any $k\in \K$ the open subset $U(k) = \Spc\K \setminus \supp(k)$ is quasi-compact.
\item[$(c)$] Any quasi-compact open subset of $\Spc\K$ is of the form $U(k)$ for some $k\in \K$.
\end{itemize}
\end{prop}
\begin{proof}
Both results can be found in \cite{BaSpec}: the first is Corollary~2.12, and the second two are the content of Proposition~2.14.
\end{proof}

In particular, it follows from the above proposition and Lemma~\ref{lem_supp_props} that the $U(k)$ for $k\in \K$ give a basis of quasi-compact open subsets for $\Spc\K$ that is closed under finite intersections as is required in the definition of a spectral space.

\begin{rem}\label{rem_noeth}
Recall, or prove as an exercise, that a space is noetherian (i.e.\ satisfies the descending chain condition for closed subsets) if and only if every open subset is quasi-compact. Combining this with (c) above we see that if $\Spc\K$ is noetherian then every open subset is of the form $U(k)$ for some $k\in \K$ and hence every closed subset is the support of some object.
\end{rem}


\subsection{The classification theorem}

We now come to the first main result of these notes, namely the abstract classification of thick tensor-ideals of $\K$ in terms of $\Spc\K$. The starting point to this story is an abstract axiomatisation of the properties of the support, based on Lemma~\ref{lem_supp_props}.

\begin{defn}
A \emph{support data} on $(\K,\otimes,\mathbf{1})$ is a pair $(X,\sigma)$ where $X$ is a topological space and $\sigma$ is an assignment associating to each object $k$ of $\K$ a closed subset $\sigma(k)$ of $X$ such that:
\begin{itemize}
\item[$(a)$] $\sigma(\mathbf{1}) = X$ and $\sigma(0) = \varnothing$;
\item[$(b)$] $\sigma(k\oplus l) = \sigma(k) \cup \sigma(l)$;
\item[$(c)$] $\sigma(\S k) = \sigma(k)$;
\item[$(d)$] for any distinguished triangle
\begin{displaymath}
k \to l \to m \to \S k
\end{displaymath}
in $\K$ there is a containment 
\begin{displaymath}
\sigma(l) \cie (\sigma(k) \cup \sigma(m));
\end{displaymath}
\item[$(e)$] $\sigma(k\otimes l) = \sigma(k) \cap \sigma(l)$.
\end{itemize}
A \emph{morphism of support data} $f\colon (X,\sigma) \to (Y,\tau)$ on $\K$ is a continuous map $f\colon X\to Y$ such that for every $k\in \K$ the equality
\begin{displaymath}
\sigma(k) = f^{-1}(\tau(k))
\end{displaymath}
is satisfied. An isomorphism of support data is a morphism $f$ of support data where $f$ is a homeomorphism.
\end{defn}

It should, given the definition of the support, come as no surprise that the pair $(\Spc\K, \supp)$ is a universal support data for $\K$.

\begin{thm}[\cite{BaSpec}*{Theorem~3.2}]
Let $\K$ be as throughout. The pair $(\Spc\K, \supp)$ is the terminal support data on $\K$. Explicitly, given any support data $(X,\sigma)$ on $\K$ there is a unique morphism of support data $f\colon (X,\sigma) \to (\Spc\K, \supp)$, i.e.\ a unique continuous map $f$ such that
\begin{displaymath}
\sigma(k) = f^{-1}\supp(k)
\end{displaymath}
for all $k\in \K$. This unique morphism is given by sending $x\in X$ to
\begin{displaymath}
f(x) = \{k\in \K\; \vert \; x\notin \sigma(k)\}.
\end{displaymath}
\end{thm}
\begin{proof}[Sketch of proof]
We give a brief sketch of the argument. We have seen in Lemma~\ref{lem_supp_props} that the pair $(\Spc\K, \supp)$ is in fact a support data for $\K$. By the same lemma it is clear that for $x\in X$ the full subcategory $f(x)$, defined as in the statement, is a proper thick tensor-ideal: it is proper by (a), closed under summands by (b), closed under $\S$ by (c), closed under extensions by (d), and a tensor-ideal by (e). To see that it is prime suppose $k\otimes l$ lies in $f(x)$. By axiom (e) for a support data we have
\begin{displaymath}
x\notin \sigma(k\otimes l) = \sigma(k)\cap \sigma(l).
\end{displaymath}
and hence $x$ must fail to lie in at least one of $\sigma(k), \sigma(l)$ implying that one of $k$ or $l$ lies in $f(x)$.

The map $f$ is a map of support data as $f(x)\in \supp k$ if and only if $k\notin f(x)$ if and only if $x\in \sigma(k)$, i.e.\
\begin{displaymath}
f^{-1}\supp(k) = \{x\in X\; \vert \; f(x)\in \supp k\} = \{x\in X\; \vert \; x\in \sigma(k)\} = \sigma(k).
\end{displaymath}
This also proves continuity of $f$ by definition of the Zariski topology on $\Spc\K$. We leave the unicity of $f$ to the reader (or it can be found as \cite{BaSpec}*{Lemma~3.3}).
\end{proof}

Before stating the promised classification theorem we need one more definition (we also provide two bonus definitions which will be useful both here and later).

\begin{defn}\label{defn_thomason}
Let $X$ be a spectral space and let $W \cie X$ be a subset of $X$. We say $W$ is \emph{specialisation closed} if for any $w\in W$ and $w' \in \overline{\{w\}}$ we have $w'\in W$. That is, $W$ is specialisation closed if it is the union of the closures of its elements. Given $w,w'$ as above we call $w'$ a \emph{specialisation} of $w$. Dually, we say $W$ is \emph{generisation closed} if given any $w' \in W$ and a $w\in X$ such that $w'\in \overline{\{w\}}$ then we have $w\in W$. In this situation we call $w$ a \emph{generisation} of $w'$.

A \emph{Thomason subset} of $X$ is a subset of the form
\begin{displaymath}
\bigcup_{\lambda \in \Lambda} V_\lambda
\end{displaymath}
where each $V_\lambda$ is a closed subset of $X$ with quasi-compact complement. Note that any Thomason subset is specialisation closed. We denote by $\Thom(X)$ the collection of Thomason subsets of $X$. It is a poset with respect to inclusion and in fact also carries the structure of a complete lattice where the join is given by taking unions.
\end{defn}

\begin{rem}\label{rem_noeth_spec}
As observed in Remark~\ref{rem_noeth} if $X$ is noetherian then every open subset is quasi-compact. Thus the Thomason subsets of $X$ are precisely the specialisation closed subsets.
\end{rem}

\begin{rem}
By Proposition~\ref{prop_spc_props} for any $k\in \K$ the subset $\supp(k)$ is a Thomason subset of $\Spc\K$, as is any union of supports of objects. Part (c) of the same proposition implies the converse, namely that any Thomason subset $V$ can be written as a union of supports of objects of $k$. This observation explains the role of Thomason subsets in the following theorem.
\end{rem}


\begin{thm}[\cite{BaSpec}*{Theorem~4.10}]\label{thm_spc_class}
The assignments
\begin{displaymath}
\sigma\colon \Thick^\otimes(\K) \to \Thom(\Spc\K), \;\; \sigma(\sfI) = \bigcup_{k\in \sfI} \supp(k)
\end{displaymath}
and
\begin{displaymath}
\tau\colon \Thom(\Spc\K) \to \Thick^\otimes(\K), \;\; \tau(V) = \{k\in \K\; \vert \; \supp(k)\cie V\},
\end{displaymath}
for $\sfI \in \Thick^\otimes(\K)$ and $V\in \Thom(\K)$ give an isomorphism of lattices
\begin{displaymath}
\Thick^\otimes(\K) \cong \Thom(\K).
\end{displaymath}
\end{thm}
\begin{proof}[Sketch of proof]
By the previous remark and the properties of the support given in Lemma~\ref{lem_supp_props} both assignments are well defined, i.e.\ $\sigma(\sfI)$ is a Thomason subset and $\tau(V)$ is a thick tensor-ideal. It is clear that both morphisms are inclusion preserving, so we just need to check that they are inverse to one another.

Consider the thick tensor-ideal $\tau\sigma(\sfI)$. It is clear that $\sfI\cie \tau\sigma(\sfI)$. The equality is proved by showing that
\begin{displaymath}
\sfI = \bigcap_{\substack{\sfP\in \Spc\K \\ \sfI\cie \sfP}} \sfP = \tau\sigma(\sfI).
\end{displaymath}
We refer to \cite{BaSpec} for the details.

On the other hand consider the Thomason subset $\sigma\tau(V)$. In this case it is clear that $\sigma\tau(V)\cie V$. By the remark preceding the theorem we know $V$ can be written as a union of supports of objects and so this containment is in fact an equality.

\end{proof}

\begin{aside}
By Stone duality the lattice $\Thom(\Spc\K)$ actually determines $\Spc\K$. So by the theorem if we know $\Thick^\otimes(\K)$ we can recover $\Spc\K$. In fact this gives another way of describing the universality of $\Spc\K$.
\end{aside}

To conclude we briefly explain what the spectrum is in each of our current running examples and then give a slightly more detailed treatment of a particular case.

\begin{ex}\label{ex_ringspec}
Let $R$ be a commutative ring. Then
\begin{displaymath}
\Spc\D^\mathrm{perf}(R) \cong \Spec R.
\end{displaymath}
This computation is due to Neeman \cite{NeeChro} (and Hopkins) in the case that $R$ is noetherian and Thomason \cite{Thomclass} in general. Applying the above theorem, in combination with Lemma~\ref{lem_small_unit_gen}, tells us that thick subcategories of $\D^\mathrm{perf}(R)$ are in bijection with Thomason subsets of $\Spec R$.
\end{ex}

\begin{ex}
As previously let $\mathrm{SH}^\mathrm{fin}$ denote the finite stable homotopy category. The description of the thick subcategories of $\mathrm{SH}^\mathrm{fin}$ by Devinatz, Hopkins, and Smith \cite{DevinatzHopkinsSmith} allows one to compute $\Spc\mathrm{SH}^\mathrm{fin}$. A nice picture of this space can be found in \cite{BaSSS}*{Corollary~9.5}.
\end{ex}

\begin{ex}
Let $G$ be a finite group and $k$ a field whose characteristic divides the order of $G$. By a result of Benson, Carlson, and Rickard \cite{BCR} we have
\begin{displaymath}
\Spc\smodu kG \cong \Proj H^\bullet(G;k)
\end{displaymath}
i.e.\ the spectrum of the stable category is the space underlying the projective scheme associated to the group cohomology ring.
\end{ex}

\subsection{An explicit example}
We now describe, in some detail but mostly without proofs, the explicit example of $\D^\mathrm{b}(\ZZ)$, the bounded derived category of finitely generated abelian groups. Our aim is to highlight some of the phenomena involved in studying the spectrum rather than to give a proof of the classification in this case.

As $\ZZ$ is a hereditary ring, i.e.\ has global dimension $1$, every complex $E\in \D^\mathrm{b}(\ZZ)$ is formal. Explicitly, this means that for every $E$ there is an isomorphism
\begin{displaymath}
E\cong \bigoplus_{i\in \ZZ} \Sigma^{-i}H^i(E).
\end{displaymath}
Thus to understand tensor-ideals in $\D^\mathrm{b}(\ZZ)$ it is sufficient to understand what happens to finitely generated abelian groups. 

Recall from Lemma~\ref{lem_small_unit_gen} that since $\ZZ$ generates $\D^\mathrm{b}(\ZZ)$ every thick subcategory is a tensor-ideal. Our first observation is that if $p\neq q$ are distinct primes in $\ZZ$ then
\begin{displaymath}
\ZZ/p\ZZ \otimes^\mathbf{L}_\ZZ \ZZ/q\ZZ \cong 0.
\end{displaymath}
If $\sfP$ is a prime tensor-ideal then it certainly contains $0$. This observation then shows that $\sfP$ must contain $\ZZ/p\ZZ$ for all but possibly one prime $p\in \ZZ$. In fact, this missing prime (or lack thereof) completely determines the prime tensor-ideal: the homeomorphism
\begin{displaymath}
\phi\colon \Spec \ZZ \to \Spc \D^\mathrm{b}(\ZZ)
\end{displaymath}
can be described explicitly as
\begin{displaymath}
\phi((p)) = \thick(\ZZ/q\ZZ\;\vert\; q\neq p).
\end{displaymath}
In particular, $\phi$ sends the $0$ ideal to the thick subcategory consisting of all complexes with torsion cohomology, which is the maximal proper thick subcategory of $\D^\mathrm{b}(\ZZ)$. As this maximal proper thick subcategory certainly contains $\phi((p))$ for all $p\neq 0$ it does indeed give the generic point of $\Spc \D^\mathrm{b}(\ZZ)$ under the Zariski topology.

On the other hand $\phi^{-1}$ sends a prime tensor-ideal $\sfP$ to
\begin{displaymath}
\phi^{-1}(\sfP) = \{n\in \ZZ \; \vert \; \cone(\ZZ \stackrel{n}\to\ZZ)\notin \sfP\},
\end{displaymath}
which one can check is really a prime ideal.

This gives a concrete illustration of the reversal of inclusions (noted in Remark~\ref{rem_gymnastics}) upon passing from prime ideals in $\ZZ$ to prime tensor-ideals in $\D^\mathrm{b}(\ZZ)$. To localise at a prime ideal $(p)\in \ZZ$ we simply invert those elements in $\ZZ\setminus (p)$. On the other hand, in $\D^\mathrm{b}(\ZZ)$, we invert these elements by killing their cones, i.e.\ taking the quotient by $\phi((p))$. The analogue is true for general commutative rings\textemdash{}localising at a smaller prime ideal inverts more elements and thus the corresponding thick subcategory in the perfect complexes is larger.

One can also use this example to gain further intuition for the support and to somewhat explain the motivation behind defining it in terms of failure to lie in prime tensor-ideals. There is an equivalence, up to idempotent completion, for each prime $p\in \ZZ$
\begin{displaymath}
\D^\mathrm{b}(\ZZ)/\phi((p)) {\to} \D^\mathrm{b}(\ZZ_{(p)})
\end{displaymath}
and one can identify, under this functor, the projection
\begin{displaymath}
\D^\mathrm{b}(\ZZ) \stackrel\pi\to \D^\mathrm{b}(\ZZ)/\phi((p))
\end{displaymath}
with localisation at $(p)$. Thus, for $E\in \D^\mathrm{b}(\ZZ)$, we have $\phi((p)) \in \supp E$ if and only if $\pi(E)\neq 0$ if and only if $E_{(p)} \neq 0$. This identifies the abstract support with the usual homological support of complexes.

\section{Generalised Rickard idempotents and supports}

In this section we use the Balmer spectrum to construct a theory of supports for triangulated categories admitting arbitrary set-indexed coproducts (so being interesting and essentially small become mutually exclusive). Our main reference is \cite{BaRickard} and again, we follow it relatively closely.

\subsection{Rigidly-compactly generated tensor triangulated categories}

We start by introducing the main players of this section, rigidly-compactly generated tensor triangulated categories. In order to define such a creature it makes sense to start with the notion of a compactly generated triangulated category.

\begin{defn}
Let $\sfT$ be a triangulated category admitting all set-indexed coproducts. We say an object $t\in \sfT$ is \emph{compact} if $\sfT(t,-)$ preserves arbitrary coproducts, i.e.\ if for any family $\{X_\lambda \; \vert \; \lambda \in \Lambda\}$ of objects in $\sfT$ the natural morphism
\begin{displaymath}
\bigoplus_{\lambda \in \Lambda} \sfT(t, X_\lambda) \to \sfT(t, \coprod_{\lambda \in \Lambda} X_\lambda)
\end{displaymath}
is an isomorphism.

We say $\sfT$ is \emph{compactly generated} if there is a set $G$ of compact objects of $\sfT$ such that an object $X\in \sfT$ is zero if and only if
\begin{displaymath}
\sfT(g, \S^i X) = 0 \;\; \text{for every} \; g\in G \; \text{and} \; i\in \ZZ.
\end{displaymath}
We denote by $\sfT^c$ the full subcategory of compact objects of $\sfT$ and note that it is an essentially small thick subcategory of $\sfT$.
\end{defn}

\begin{ex}\label{ex_cg}
Let us give some examples of compactly generated triangulated categories which will be used as illustrations throughout.
\begin{itemize}
\item[(1)] Let $R$ be a ring and denote by $\D(R)$ the unbounded derived category of $R$. Then $\D(R)$ is compactly generated and $R$ is a compact generator for $\D(R)$.
\item[(2)] Let $\mathrm{SH}$ denote the stable homotopy category. The sphere spectrum $S^0$ is a compact generator for $\mathrm{SH}$ and hence it is compactly generated.
\item[(3)] Let $G$ be a finite group and let $k$ be a field whose characteristic divides the order of $G$. Then $\sModu kG$, the stable category of aribtrary $kG$-modules is a compactly generated triangulated category.
\end{itemize}
\end{ex}

\begin{defn}
A \emph{compactly generated tensor triangulated category} is a triple $(\sfT, \otimes, \mathbf{1})$ where $\sfT$ is a compactly generated triangulated category, and $(\otimes, \mathbf{1})$ is a symmetric monoidal structure on $\sfT$ such that the tensor product $\otimes$ is a coproduct preserving exact functor in each variable and the compact objects $\sfT^c$ form a tensor subcategory. In particular, we require that the unit $\mathbf{1}$ is compact.
\end{defn}

\begin{rem}
We will frequently supress the tensor product and unit and just refer to $\sfT$ as a compactly generated tensor triangulated category.
\end{rem}

\begin{rem}
By Brown representability \cite{NeeGrot}*{Theorem~3.1}, and the assumption that the tensor product is coproduct preserving, it is automatic that $\sfT$ is closed symmetric monoidal\textemdash{}for each $X\in \sfT$ the functor $X\otimes -$ has a right adjoint which we denote by $\hom(X,-)$.
\end{rem}

Finally, we come to the combination of hypotheses that we can get the most mileage out of.

\begin{defn}
A \emph{rigidly-compactly generated tensor triangulated category} $\sfT$ is a compactly generated tensor triangulated category such that the full subcategory $\sfT^c$ of compact objects is rigid. Explicitly, not only is $\sfT^c$ a tensor subcategory of $\sfT$, but it is closed under the internal hom (which exists by the previous remark) and is rigid in the sense of Definition~\ref{defn_rigid}.
\end{defn}

\begin{ex}
Each of the triangulated categories covered in Example~\ref{ex_cg} (with the proviso in (1) that the ring is commutative) is a rigidly-compactly generated tensor triangulated category via the left derived tensor product, smash product, and tensor product over the ground field with the diagonal action respectively.
\end{ex}

\subsection{Localising sequences and smashing localisations}
We now review some definitions concerning the particular subcategories of a rigidly-compactly generated tensor triangulated category that we will wish to consider. Of course some of the definitions we make are valid more generally, but we will restrict ourselves to what we need. Throughout we fix a rigidly-compactly generated tensor triangulated category $\sfT$.

\begin{defn}
Let $\sfL$ be a triangulated subcategory of $\sfT$. The subcategory $\sfL$ is \emph{localising} if it is closed under arbitrary coproducts in $\sfT$. Dually, it is \emph{colocalising} if it is closed under arbitrary products in $\sfT$. Put another way, a full subcategory of $\sfT$ is localising if it is closed under suspensions, cones, and coproducts and colocalising if it is closed under suspensions, cones, and products.

Let $\sfI$ be a localising subcategory of $\sfT$. We call $\sfI$ a \emph{localising tensor-ideal} if for any $X\in \sfT$ and $Y\in \sfI$ the object $X\otimes Y$ lies in $\sfI$.

Given a family of objects $S\cie \sfT$ we denote by $\loc(S)$ (resp.\ $\loc^\otimes(S)$) the smallest localising subcategory (resp.\ localising tensor-ideal) of $\sfT$ containing $S$ and by $\Loc(\sfT)$ (resp.\ $\Loc^\otimes(\sfT)$) the collection of all localising subcategories (resp.\ localising tensor-ideals) of $\sfT$.
\end{defn}

\begin{rem}
There is, \emph{a priori}, no reason that either $\Loc(\sfT)$ or $\Loc^\otimes(\sfT)$ should be sets. In fact, whether or not these collections can form proper classes in the given situation is a longstanding open problem.
\end{rem}

\begin{rem}
By \cite{NeeCat}*{Proposition~1.6.8} every localising subcategory and every colocalising subcategory of $\sfT$ is automatically closed under direct summands, i.e.\ (co)localising implies thick.
\end{rem}

\begin{rem}\label{rem_gen}
One can show, see for instance \cite{NeeGrot}*{Lemma~3.2}, that a set $G$ of compact objects in $\sfT$ is a set of compact generators for $\sfT$ if and only if $\loc(G) = \sfT$.
\end{rem}

This last remark is very important in practice. For instance, it is a key point in proving the following elementary but useful lemma. Perhaps more importantly, the proof illustrates an argument which appears frequently in this subject.

\begin{lem}\label{lem_big_rigid}
Let $\sfT$ be a rigidly-compactly generated tensor triangulated category. Then, for $t\in \sfT^c$ the isomorphism of endofunctors of $\sfT^c$
\begin{displaymath}
t^\vee\otimes (-) \stackrel{\a}{\to} \hom(t,-)
\end{displaymath}
extends to an isomorphism of the corresponding endofunctors of $\sfT$. That is, for any $X\in \sfT$ we have $t^\vee\otimes X \cong \hom(t,X)$. In particular, $\hom(t,-)$ preserves coproducts and $t\otimes (-)$ is left and right adjoint to $t^\vee\otimes (-)$ when viewed as endofunctors of $\sfT$.
\end{lem}
\begin{proof}
As in Definition~\ref{defn_rigid} we can define a natural transformation
\begin{displaymath}
\a\colon t^\vee \otimes (-) \to \hom(t,-)
\end{displaymath}
of functors $\sfT \to \sfT$. Define a full subcategory of $\sfT$ as follows
\begin{displaymath}
\sfS = \{X\in \sfT \; \vert \; \a_X\colon t^\vee\otimes X \to \hom(t,X) \; \text{is an isomorphism}\}.
\end{displaymath}
As both functors involved preserve suspensions and coproducts, and $\a$ is also compatible with suspensions and coproducts, we see that $\sfS$ is closed under coproducts and suspensions. Given a triangle $X\to Y \to Z \to \S X$ with $X,Y \in \sfS$ the naturality of $\a$ guarantees the commutativity of the following diagram
\begin{displaymath}
\xymatrix{
t^\vee\otimes X \ar[r] \ar[d]^-{\a_X} & t^\vee\otimes Y \ar[r] \ar[d]^-{\a_Y} & t^\vee\otimes Z \ar[r] \ar[d]^-{\a_Z} & \Sigma (t^\vee\otimes X) \ar[d]^-{\a_{\Sigma X}}\\
\hom(t,X) \ar[r] & \hom(t,Y) \ar[r] & \hom(t,Z) \ar[r] & \Sigma\hom(t,X)
}
\end{displaymath}
As both functors are exact the rows are triangles and so since $\a_X$ and $\a_Y$ are isomorphisms it follows that $\a_Z$ must also be an isomorphism and hence $Z\in \sfS$. Thus $\sfS$ is a localising subcategory and, by rigidity of $\sfT^c$, we have $\sfT^c \cie \sfS$. It then follows from \cite{NeeGrot}*{Lemma~3.2} that $\sfS = \sfT$ which proves the first assertion of the Lemma; the rest of the statements are immediate.
\end{proof}

Another illustration is given by the following analogue of Lemma~\ref{lem_small_unit_gen}.

\begin{lem}\label{lem_big_unit_gen}
Let $\sfT$ be a rigidly-compactly generated tensor triangulated category. If $\mathbf{1}$ is a compact generator for $\sfT$, i.e.\ if $\sfT = \loc(\mathbf{1})$, then every localising subcategory of $\sfT$ is a localising tensor-ideal.
\end{lem}
\begin{proof}
Exercise.
\end{proof}

Localising subcategories are not imaginatively named\textemdash{}they are precisely the kernels of localisation functors. We will only give a brief outline of the formalism that we need and do not go into the details of forming Verdier quotients. There are many excellent sources for further information on this topic such as \cite{NeeCat}, \cite{NeeHolim}, and \cite{KrLoc}.

\begin{defn}\label{defn_locsequence}
A \emph{localisation sequence} is a diagram
\begin{displaymath}
\xymatrix{
\sfL \ar[r]<0.5ex>^-{i_*} \ar@{<-}[r]<-0.5ex>_-{i^!} & \sfT \ar[r]<0.5ex>^-{j^*} \ar@{<-}[r]<-0.5ex>_-{j_*} & \sfC
}
\end{displaymath}
where $i^!$ is right adjoint to $i_*$ and $j_*$ is right adjoint to $j^*$, both $i_*$ and $j_*$ are fully faithful and hence embed $\sfL$ and $\sfC$ as a localising and a colocalising subcategory respectively, and we have equalities
\begin{displaymath}
(i_*\sfL)^\perp = j_*\sfC \quad \text{and} \quad {}^\perp(j_*\sfC) = i_*\sfL,
\end{displaymath}
where
\begin{displaymath}
(i_*\sfL)^\perp = \{Y\in \sfT \; \vert \; \sfT(i_*L, Y) = 0 \;\; \text{for all} \; L\in \sfL\}
\end{displaymath}
and
\begin{displaymath}
{}^\perp(j_*\sfC) = \{Y\in \sfT \; \vert \; \sfT(Y,j_*C) = 0 \;\; \text{for all} \; C\in \sfC\}.
\end{displaymath}
We refer to the composite $i_*i^!$ as the \emph{acyclisation} functor and $j_*j^*$ as the \emph{localisation} functor corresponding to this localisation sequence.

We will frequently abuse the notation in such situations and identify $\sfL$ and $\sfC$ with their images under the fully faithful functors $i_*$ and $j_*$.
\end{defn}

The existence of a localisation sequence provides a great deal of information and has many consequences. The following proposition lists a few of the ones we will need. We suggest perusing the references given before the above definition for more details as well as proofs of the statements made below (there are, in addition, relevant references in \cite{BaRickard}*{Theorem~2.6} where these statements also appear).

\begin{prop}\label{prop_loc_seq}
If we have a localisation sequence as in the definition then the following statements hold.
\begin{itemize}
\item[$(a)$] The composites $j^*i_*$ and $i^!j_*$ are zero. Moreover, the kernel of $j^*$ is precisely $\sfL$.
\item[$(b)$] The composite 
\begin{displaymath}
\sfC \stackrel{j_*}{\to} \sfT \to \sfT/\sfL
\end{displaymath}
is an equivalence. In particular, the Verdier quotient $\sfT/\sfL$ is locally small and the canonical projection $\sfT \to \sfT/\sfL$ has a right adjoint.
\item[$(c)$] For every $X\in \sfT$ there is a distinguished triangle
\begin{displaymath}
i_*i^!X \to X \to j_*j^*X \to \S i_*i^!X.
\end{displaymath}
These triangles are functorial and unique in the sense that given any distinguished triangle
\begin{displaymath}
X' \to X \to X'' \to \S X \quad \text{with} \; X'\in \sfL \; \text{and} \; X''\in \sfC
\end{displaymath}
there are unique isomorphisms $X' \cong i_*i^!X$ and $X'' \cong j_*j^*X$.
\item[$(d)$] A localisation sequence is completely determined by either of the pairs of adjoint functors $(i_*,i^!)$ or $(j^*,j_*)$.
\end{itemize}
\end{prop}

In nature localising sequences are actually rather easy to come by as a consequence of the following form of Brown representability.

\begin{thm}\label{thm_brown}
Let $\sfT$ be a compactly generated triangulated category, $L\cie \sfT$ a set of objects, and set $\sfL = \loc(L)$. Then the inclusion $i_*\colon \sfL \to \sfT$ admits a right adjoint. In particular, $\sfL$ fits into a localisation sequence
\begin{displaymath}
\xymatrix{
\sfL \ar[r]<0.5ex>^-{i_*} \ar@{<-}[r]<-0.5ex>_-{i^!} & \sfT \ar[r]<0.5ex>^-{j^*} \ar@{<-}[r]<-0.5ex>_-{j_*} & \sfT/\sfL
}
\end{displaymath}
\end{thm}
\begin{proof}
This follows from the results in Section~7.2 of \cite{KrLoc} combined with \cite{NeeCat}*{Theorem~8.4.4}.
\end{proof}

By definition, given any localisation sequence as above, the functors $i_*$ and $j^*$ preserve coproducts by virtue of being left adjoints. The condition that the right adjoints $i^!$ and $j_*$ also preserve coproducts turns out to be very strong and very useful.

\begin{defn}
A localisation sequence as in Definition~\ref{defn_locsequence} is \emph{smashing} if $i^!$ (or equivalently $j_*$) preserves coproducts. In this case we call $\sfL$ a \emph{smashing subcategory} of $\sfT$. If $\sfL$ is, moreover, a tensor-ideal we call it a \emph{smashing tensor-ideal}.
\end{defn}

In a smashing localisation sequence $\sfC$ is also, under the embedding $j_*$, a localising subcategory of $\sfT$. The standard source of smashing localisation sequences is the following well known lemma.

\begin{lem}\label{lem_smash}
Let $\sfT$ be a compactly generated triangulated category and $S\cie \sfT^c$ a set of compact objects. Then $\sfS = \loc(S)$ is a smashing subcategory, i.e.\ the inclusion $i_*\colon \sfS \to \sfT$ admits a coproduct preserving right adjoint $i^!$.
\end{lem}
\begin{proof}
We know from Theorem~\ref{thm_brown} that $i_*$ admits a right adjoint $i^!$. It follows from Thomason's localisation theorem \cite{NeeGrot}*{Theorem~2.1} together with \cite{NeeGrot}*{Theorem~5.1} that $i^!$ preserves coproducts.
\end{proof}

\subsection{Generalised Rickard idempotents and supports}
Throughout this section we will assume $\sfT$ is a rigidly-compactly generated tensor triangulated category such that $\Spc\sfT^c$, the spectrum of the compact objects, is a noetherian topological space. Recall from Remark~\ref{rem_noeth_spec} that in this case the Thomason subsets of $\Spc\sfT^c$ are precisely the specialisation closed subsets. 

\begin{aside}
The assumption that $\Spc \sfT^c$ is noetherian is not necessary, but the situation is significantly more complicated without this assumption. If the spectrum is not noetherian then one can not necessarily define a tensor-idempotent, as in Definition~\ref{defn_gen_ptfunctors2}, for every point. However, the main ideas, as introduced in \cite{BaRickard} and discussed below, do work to some extent without the noetherian hypotheses and as shown in \cite{StevensonAbsFlat} one always has enough information to recover the supports of compact objects via the support of Definition~\ref{defn_bigsupport}.
\end{aside}

We now introduce the objects which allow us to extend the support defined in the first section for objects of $\sfT^c$ to arbitrary objects of $\sfT$. The construction was motivated by Rickard's work in modular representation theory \cite{Rickardidempotent} and the work of Hovey, Palmieri, and Strickland \cite{HPS} on the correct axiomatic framework in which to perform such constructions. 

We produce the objects we desire, named \emph{Rickard idempotents}, via the following key construction based on certain smashing localisations. This result is now somewhat standard, but we include some details. Further information and references can be found in \cite{BaRickard}.

\begin{thm}
Let $\sfT$ be a rigidly-compactly generated tensor triangulated category, $S\cie \sfT^c$ a set of compact objects, and set $\sfS = \loc^\otimes(S)$. Consider the corresponding smashing localisation sequence
\begin{displaymath}
\xymatrix{
\sfS \ar[r]<0.5ex>^-{i_*} \ar@{<-}[r]<-0.5ex>_-{i^!} & \sfT \ar[r]<0.5ex>^-{j^*} \ar@{<-}[r]<-0.5ex>_-{j_*} & \sfS^\perp
}
\end{displaymath}
Then:
\begin{itemize}
\item[$(a)$] $\sfS^\perp$ is a localising tensor-ideal;
\item[$(b)$] there are isomorphisms of functors
\begin{displaymath}
i_*i^!\mathbf{1} \otimes (-) \cong i_*i^! \quad \text{and} \quad j_*j^*\mathbf{1}\otimes (-) \cong j_*j^*;
\end{displaymath}
\item[$(c)$] the objects $i_*i^!\mathbf{1}$ and $j_*j^*\mathbf{1}$ satisfy
\begin{displaymath}
i_*i^!\mathbf{1} \otimes i_*i^!\mathbf{1} \cong i_*i^!\mathbf{1}, \;\; j_*j^*\mathbf{1}\otimes j_*j^*\mathbf{1} \cong j_*j^*\mathbf{1}, \;\; \text{and} \;\; i_*i^!\mathbf{1} \otimes j_*j^*\mathbf{1} \cong 0,
\end{displaymath}
i.e.\ they are \emph{tensor idempotent} and tensor to $0$.
\end{itemize}
\end{thm}
\begin{proof}
The claimed smashing localisation sequence exists by Lemma~\ref{lem_smash}. The only hitch is that one needs to know $\sfS$ is in fact generated by objects of $\sfT^c$. In fact $\loc^\otimes(S) = \loc(\thick^\otimes(S))$ and this can be proved directly or deduced from \cite{StevensonActions}*{Lemma~3.8} and is left as an exercise. We begin by proving (a); this is the main point and the other statements follow in a straightforward manner from abstract properties of localisations.

Consider the following full subcategory of $\sfT$
\begin{displaymath}
\sfM = \{ X\in \sfT \; \vert \; X\otimes \sfS^\perp \cie \sfS^\perp\}.
\end{displaymath}
As $\otimes$ is exact and coproduct preserving in each variable and $\sfS^\perp$ is localising it is straightforward to verify that $\sfM$ is a localising subcategory. Let $t$ be a compact object of $\sfT$ and $Y\in \sfS^\perp$. We have isomorphisms for any $Z\in \sfS$
\begin{align*}
\sfT(Z, t\otimes Y) &\cong \sfT(Z, \hom(t^\vee, Y)) \\
&\cong \sfT(Z\otimes t^\vee, Y) \\
&= 0
\end{align*}
where the first isomorphism is via Lemma~\ref{lem_big_rigid}, the second is by adjunction, and the final equality holds as $\sfS$ is a tensor ideal, so $Z\otimes t^\vee\in \sfS$, together with the fact that $Y\in \sfS^\perp$. As $Z\in \sfS$ was arbitrary this shows $t\otimes Y\in \sfS^\perp$, i.e.\ $\sfT^c \cie \sfM$. It follows from Remark~\ref{rem_gen} that $\sfM = \sfT$ which says precisely that $\sfS^\perp$ is a tensor-ideal.

Now we show (b). Consider the localisation triangle
\begin{displaymath}
i_*i^!\mathbf{1} \to \mathbf{1} \to j_*j^*\mathbf{1} \to
\end{displaymath}
from Proposition~\ref{prop_loc_seq}~(c). Given $X\in \sfT$ we can tensor this triangle with $X$ to obtain the distinguished triangle
\begin{displaymath}
i_*i^!\mathbf{1}\otimes X \to X \to j_*j^*\mathbf{1}\otimes X \to
\end{displaymath}
As both $\sfS$ and $\sfS^\perp$ are tensor-ideals the leftmost and rightmost terms of this latter triangle lie in $\sfS$ and $\sfS^\perp$ respectively. Uniqueness and functoriality of localisation triangles, also observed in (c) of the aforementioned proposition, then guarantees unique isomorphisms
\begin{displaymath}
i_*i^!X \cong i_*i^!\mathbf{1}\otimes X \quad \text{and} \quad j_*j^*X \cong j_*j^*\mathbf{1}\otimes X
\end{displaymath}
which can be assembled to the desired isomorphisms of functors. We leave (c) as an exercise so the interested reader can familiarise themselves with the properties of the localisation and acyclisation functors ($j_*j^*$ and $i_*i^!$ respectively) associated to localisation sequences.
\end{proof}

We now apply this proposition to certain subcategories arising from Theorem~\ref{thm_spc_class}. Let $\mathcal{V}$ be a specialisation closed subset of $\Spc\sfT^c$. Recall that $\tau(\mathcal{V})$ denotes the associated thick tensor-ideal
\begin{displaymath}
\tau(\mathcal{V}) = \{t\in \sfT^c \; \vert \; \supp t \cie \mathcal{V}\}.
\end{displaymath}
We set
\begin{displaymath}
\mathit{\Gamma}_\mathcal{V}\sfT = \loc(\tau(\mathcal{V})).
\end{displaymath}
By Lemma~\ref{lem_smash} there is an associated smashing localisation sequence
\begin{displaymath}
\xymatrix{
\mathit{\Gamma}_\mathcal{V}\sfT \ar[r]<0.5ex> \ar@{<-}[r]<-0.5ex> & \sfT \ar[r]<0.5ex> \ar@{<-}[r]<-0.5ex> & L_\mathcal{V}\sfT 
}
\end{displaymath}
and we denote the corresponding acyclisation and localisation functors by $\mathit{\Gamma_\mathcal{V}}$ and $L_\mathcal{V}$ respectively. By \cite{BaRickard}*{Theorem~4.1} $\mathit{\Gamma}_\mathcal{V}\sfT$ is not only a smashing subcategory but a smashing tensor-ideal. Thus applying the preceding theorem yields tensor idempotents $\mathit{\Gamma}_\mathcal{V}\mathbf{1}$ and $L_\mathcal{V}\mathbf{1}$ which give rise to the acyclisation and localisation functors by tensoring. It is these idempotents that are used to define the support; the intuition is that, for $X\in \sfT$, $\mathit{\Gamma}_\mathcal{V}\mathbf{1} \otimes X$ is the ``piece of $X$ supported on $\mathcal{V}$'' and $L_\mathcal{V}\mathbf{1}\otimes X$ is the ``piece of $X$ supported on the complement of $\mathcal{V}$''.

Let us make a quick comment on notation before continuing. From this point forward we will generally denote points of $\Spc\sfT^c$ by $x,y,\ldots$ rather than in a notation suggestive of prime tensor-ideals as in the last section. This is due to the fact that we will work with $\Spc\sfT^c$ as an abstract topological space which, together with $\supp$, forms a support data rather than explicitly with prime tensor-ideals.

\begin{defn}\label{defn_gen_ptfunctors}
For every $x \in \Spc \sfT^c$ we define subsets of the spectrum
\begin{displaymath}
\mathcal{V}(x) = \overline{\{x\}}
\end{displaymath}
and
\begin{displaymath}
\mathcal{Z}(x) = \{y\in \Spc \sfT^c \; \vert \; x\notin \mathcal{V}(y)\}.
\end{displaymath}
Both of these subsets are specialization closed and hence Thomason as we have assumed $\Spc\sfT^c$ is noetherian. We note that
\begin{displaymath}
\mathcal{V}(x) \setminus (\mathcal{Z}(x)\cap \mathcal{V}(x)) = \{x\}
\end{displaymath}
i.e.\ these two Thomason subsets let us pick out the point $x$.
\end{defn}

\begin{defn}\label{defn_gen_ptfunctors2}
Let $x$ be a point of $\Spc \sfT^c$. We define a tensor-idempotent
\begin{displaymath}
\mathit{\Gamma}_x\mathbf{1}= (\mathit{\Gamma}_{\mathcal{V}(x)}\mathbf{1} \otimes L_{\mathcal{Z}(x)}\mathbf{1}).
\end{displaymath}
Following the intuition above, for an object $X\in \sfT$, the object $\mathit{\Gamma}_x\mathbf{1}\otimes X$ is supposed to be the ``piece of $X$ which lives only over the point $x\in \Spc\sfT^c$''.
\end{defn}

\begin{rem}\label{rem_bik6.2_uber}
Given any Thomason subsets $\mathcal{V}$ and $\mathcal{W}$ of $\Spc \sfT^c$ such that 
\begin{displaymath}
\mathcal{V} \setminus (\mathcal{V}\intersec \mathcal{W}) = \{x\}
\end{displaymath}
we can define a similar object by forming the tensor product $\mathit{\Gamma}_\mathcal{V}\mathbf{1} \otimes L_\mathcal{W}\mathbf{1}$. By \cite{BaRickard}*{Corollary~7.5} any such object is uniquely isomorphic to $\mathit{\Gamma}_x\mathbf{1}$.
\end{rem}

\begin{ex}\label{ex_ring_ex}
As a brief respite from the abstraction let us provide a detailed example describing what these idempotents look like in $\D(R)$, the derived category of a  commutative noetherian ring.

Given an element $f\in R$ we define the \emph{stable Koszul complex} $K_\infty(f)$ to be the complex concentrated in degrees 0 and 1
\begin{displaymath}
\cdots \to 0 \to R \to R_f \to 0 \to \cdots
\end{displaymath}
where the only non-zero morphism is the canonical map to the localisation. Given a sequence of elements $\mathbf{f} = \{f_1,\ldots,f_n\}$ of $R$ we set
\begin{displaymath}
K_\infty(\mathbf{f}) = K_\infty(f_1) \otimes \cdots \otimes K_\infty(f_n).
\end{displaymath}
There is a canonical morphism $K_\infty(\mathbf{f}) \to R$ which is clearly a degreewise split epimorphism. We define the \emph{\^Cech complex} of $\mathbf{f}$ to be the suspension of the kernel of this map. As the stable Koszul complex and the desuspension of the \^Cech complex fit into a degreewise split short exact sequence there is a triangle
\begin{displaymath}
K_\infty(\mathbf{f}) \to R \to \check{C}(\mathbf{f}) \to \S K_\infty(\mathbf{f}).
\end{displaymath}
Explicitly we have
\begin{displaymath}
\check{C}(\mathbf{f})^t = \bigoplus_{i_0<\cdots<i_t}R_{f_{i_0}\cdots f_{i_t}}
\end{displaymath}
for $0\leq t \leq n-1$ and $K_\infty(\mathbf{f})$ is essentially the same complex desuspended and with $R$ in degree 0. Given an ideal $I$ of $R$ we can define $K_\infty(I)$ and $\check{C}(I)$ by choosing a set of generators for $I$. It turns out the complex obtained is independent of the choice of generators up to quasi-isomorphism. We observe that both of these complexes are bounded complexes of flat $R$-modules and hence are \emph{K-flat}: tensoring, at the level of chain complexes, with either complex preserves quasi-isomorphisms.

Using these complexes we can give the following well known explicit descriptions of the Rickard idempotents corresponding to some specialization closed subsets. For an ideal $I\cie R$ and $\mathfrak{p}\in \Spec R$ a prime ideal there are natural isomorphisms in $D(R)$:
\begin{itemize}
\item[(1)] $\mathit{\Gamma}_{\mathcal{V}(I)}R \iso K_{\infty}(I)$;
\item[(2)] $L_{\mathcal{V}(I)}R \iso \check{C}(I)$;
\item[(3)] $L_{\mathcal{Z}(\mathfrak{p})}R \iso R_\mathfrak{p}$.
\end{itemize}
In particular, the objects $\mathit{\Gamma}_\mathfrak{p}R = \mathit{\Gamma}_{\mathcal{V}(\mathfrak{p})}R\otimes L_{\mathcal{Z}(\mathfrak{p})}R$, which will give rise to supports on $D(R)$, are naturally isomorphic to $K_{\infty}(\mathfrak{p})\otimes R_\mathfrak{p}$. Statements (1) and (2) are already essentially present in \cite{HartshorneLC}. In the form stated here they are special cases of \cite{GreenleesTate}*{Lemma 5.8}. The third statement can be proved by noting that the full subcategory of complexes with homological support in $\mathcal{U}(\mathfrak{p}) = \Spec R \setminus \mathcal{Z}(\mathfrak{p})$ is the essential image of the inclusion of $D(R_\mathfrak{p})$ and then appealing to the fact that, up to direct summands,
\begin{displaymath}
\D^\mathrm{perf}(R)/\tau(\mathcal{Z}(\mathfrak{p})) \cong \D^\mathrm{perf}(R_\mathfrak{p}).
\end{displaymath}
\end{ex}

We now come to the main definition of this section.

\begin{defn}\label{defn_bigsupport}
For $X\in \sfT$ we define the \emph{support} of $X$ to be
\begin{displaymath}
\supp X = \{x\in \Spc\sfT^c \; \vert \; \mathit{\Gamma}_x\mathbf{1}\otimes X \neq 0\}.
\end{displaymath}
\end{defn}

We do not introduce notation to distinguish this notion of support, applied to compact objects, from the one for the first section as they agree. The following proposition records this fact as well as several other important properties of the support and the tensor-idempotents we have defined.

\begin{prop}[\cite{BaRickard}*{7.17, 7.18}]\label{prop_supp_prop}
The support defined above satisfies the following properties.
\begin{itemize}
\item[$(a)$] The two notions of support coincide for any compact object of $\sfT$.
\item[$(b)$] For any set-indexed family of objects $\{X_\lambda\; \vert \; \lambda \in \Lambda\}$ we have
\begin{displaymath}
\supp (\coprod_{\lambda \in \Lambda} X_\lambda) = \bigcup_{\lambda \in \Lambda} \supp X_\lambda.
\end{displaymath}
\item[$(c)$] For every $X\in \sfT$ $\supp(\S X) = \supp(X)$.
\item[$(d)$] Given a distinguished triangle $X \to Y \to Z \to$ in $\sfT$ we have
\begin{displaymath}
\supp(Y) \cie \supp(X) \cup \supp(Z).
\end{displaymath}
\item[$(e)$] For any $X,Y \in \sfT$ we have $\supp(X\otimes Y)\cie \supp(X)\cap \supp(Y)$.
\item[$(f)$] For any $X\in \sfT$ and Thomason subset $\mathcal{V}\cie \Spc\sfT^c$ we have
\begin{align*}
&\supp (\mathit{\Gamma}_\mathcal{V}\mathbf{1}\otimes X) = (\supp X) \intersec \mathcal{V} \\
&\supp (L_\mathcal{V}\mathbf{1}\otimes X) = (\supp X) \intersec (\Spc \sfT^c \setminus \mathcal{V}).
\end{align*}
\end{itemize}
\end{prop}

\begin{rem}
Note that (e) is weaker than the corresponding statement given in Lemma~\ref{lem_supp_props}. It is also worth noting that we have not mentioned that an object $X\in \sfT$ is $0$ if and only if $\supp X = \varnothing$. In fact this latter statement is much more subtle for rigidly-compactly generated tensor triangulated categories. We will show, with an additional hypothesis, that it does indeed hold though when we discuss the local-to-global principle in the next section.
\end{rem}

We defer any detailed examples to the end of the next section. One can, using just the definition of the support and these properties, already make interesting statements but life is much easier once we have the local-to-global principle. 

\section{Tensor actions and the local-to-global principle}
In the last section we saw how to define a notion of support for rigidly-compactly generated tensor triangulated categories using the spectrum of the compacts. In this section we will explain a relative version, where one defines supports in terms of the action of a rigidly-compactly generated tensor triangulated category. What we have seen is then the special case of such a category acting on itself in the obvious way. The main reference for this section is \cite{StevensonActions}.

Continuing our analogy from the first section: if a tensor triangulated category is viewed as a strange sort of ring, we will now study the basics of the corresponding module theory.

\subsection{Actions and submodules}
We define here the notion of left action and express a sinistral bias by only considering left actions. We will omit the ``left'' and refer to them simply as actions. One can, of course, define right actions in the analogous way, although since we deal with symmetric monoidal categories this is mostly cosmetic.

We fix throughout this section a rigidly-compactly generated tensor triangulated category $\sfT$ and a compactly generated triangulated category $\K$. We will continue to assume throughout that $\Spc\sfT^c$ is noetherian. Again a lot of the theory can be developed without this hypothesis, but we avoid the related technicalities for the sake of simplicity.

\begin{defn}\label{defn_action}
A \emph{(left) action} of $\sfT$ on $\K$ is a functor 
\begin{displaymath}
*\colon \sfT\times \K \to \K
\end{displaymath}
which is exact in each variable, i.e.\ for all $X\in \sfT$ and $A\in \K$ the functors $X*(-)$ and $(-)*A$ are exact, together with natural isomorphisms
\begin{displaymath}
a_{X,Y,A}\colon (X\otimes Y)*A \stackrel{\sim}{\to} X*(Y*A)
\end{displaymath}
and
\begin{displaymath}
l_A\colon \mathbf{1}*A \stackrel{\sim}{\to} A
\end{displaymath}
for all $X,Y\in \sfT$, $A\in \K$, compatible with the exactness hypotheses on $(-)*(-)$ and satisfying the following compatibility conditions:
\begin{itemize}
\item[(1)] The associator $a$ makes the following diagram commute for all $X,Y,Z$ in $\sfT$ and $A$ in $\K$
\begin{displaymath}
\xymatrix{
& X*(Y*(Z* A)) & \\
X*((Y \otimes Z)*A) \ar[ur]^{X*a_{Y,Z,A}} & & (X\otimes Y)*(Z*A) \ar[ul]_{a_{X,Y,Z*A}} \\
(X\otimes(Y\otimes Z))*A \ar[u]^{a_{X,Y\otimes Z,A}} && ((X\otimes Y)\otimes Z)*A \ar[u]_{a_{X\otimes Y, Z,A}} \ar[ll]
}
\end{displaymath}
where the bottom arrow is the associator of $(\sfT,\otimes,\mathbf{1})$.
\item[(2)] The unitor $l$ makes the following squares commute for every $X$ in $\sfT$ and $A$ in $\K$
\begin{displaymath}
\xymatrix{
X* (\mathbf{1}*A) \ar[r]^(0.6){X * l_A} & X*A \ar[d]^{1_{X*A}} \\
(X\otimes \mathbf{1})*A \ar[u]^{a_{X,\mathbf{1},A}} \ar[r] & X*A
}
\qquad
\xymatrix{
\mathbf{1}*(X*A) \ar[r]^(0.6){l_{X*A}} & X*A \ar[d]^{1_{X*A}} \\
(\mathbf{1}\otimes X)*A \ar[u]^{a_{\mathbf{1},X,A}} \ar[r] & X*A
}
\end{displaymath}
where the bottom arrows are the right and left unitors of $(\sfT,\otimes,\mathbf{1})$.
\item[(3)] For every $A$ in $\K$ and $r,s \in \ZZ$ the diagram
\begin{displaymath}
\xymatrix{
\S^r \mathbf{1}*\S^s A \ar[r]^(0.6)\sim \ar[d]_{\wr} & \S^{r+s} A \ar[d]^{(-1)^{rs}} \\
\S^r (\mathbf{1}* \S^s A) \ar[r]_(0.6)\sim & \S^{r+s}A
}
\end{displaymath}
is commutative, where the left vertical map comes from exactness in the first variable of the action, the bottom horizontal map is the unitor, and the top map is given by the composite
\begin{displaymath}
\S^r\mathbf{1}*\S^s A \to \S^s(\S^r\mathbf{1} * A) \to \S^{r+s}(\mathbf{1}*A) \stackrel{l}{\to} \S^{r+s}A
\end{displaymath}
whose first two maps use exactness in both variables of the action.
\item[(4)] The functor $*$ distributes over coproducts. Explicitly, for families of objects $\{X_i\}_{i\in I}$ in $\sfT$ and $\{A_j\}_{j\in J}$ in $\K$, and $X$ in $\sfT$, $A$ in $\K$ the canonical maps
\begin{displaymath}
\coprod_i (X_i * A) \stackrel{\sim}{\to} (\coprod_i X_i)* A
\end{displaymath}
and
\begin{displaymath}
\coprod_j (X*A_j) \stackrel{\sim}{\to} X*(\coprod_j A_j)
\end{displaymath}
are isomorphisms.
\end{itemize}
\end{defn}

\begin{rem}
Given composable morphisms $f,f'$ in $\sfT$ and $g,g'$ in $\K$ one has
\begin{displaymath}
(f'*g')(f*g) = (f'f*g'g)
\end{displaymath}
by functoriality of $\sfT \times \K \stackrel{*}{\to} \K$.

We also note it follows easily from the definition that both $0_\sfT*(-)$ and $(-)*0_\K$ are isomorphic to the zero functor. Indeed if $X\in \sfT$ then one has a triangle
\begin{displaymath}
\xymatrix{
X \ar[r]^-{1_X} & X \ar[r] & 0 \ar[r] & \Sigma X
}
\end{displaymath}
and applying $-*A$ for $A\in \K$ we get a triangle
\begin{displaymath}
\xymatrix{
X*A \ar[r]^-{1_{X*A}} & X*A \ar[r] & 0*A \ar[r] & \Sigma X*A.
}
\end{displaymath}
Thus $0*A \cong 0$ for every $A$ in $\K$ and the argument for $(-)*0_\K$ is similar.
\end{rem}

We view $\K$ as a module over $\sfT$ and from now on we will use the terms module and action interchangeably. There are of course, depending on the context, natural notions of $\sfT$-submodule.

\begin{defn}\label{defn_submodule}
Let $\sfL\cie \K$ be a localising (thick) subcategory of $\K$. We say $\sfL$ is a localising (thick) \emph{$\sfT$-submodule} of $\K$ if the functor
\begin{displaymath}
\sfT\times \sfL \stackrel{*}{\to} \K
\end{displaymath}
factors via $\sfL$, i.e.\ $\sfL$ is closed under the action of $\sfT$. In the case that $\K = \sfT$ acts on itself by $\otimes$ this recovers the notion of a localising (thick) tensor-ideal of $\sfT$. We mean the obvious things by a smashing or compactly generated (by compact objects in the ambient category) submodule.
\end{defn}

\begin{notation}\label{not_smallestsubcat}
For a collection of objects $\mathcal{A}$ in $\K$ we denote, as before, by $\loc(\mathcal{A})$ the smallest localising subcategory containing $\mathcal{A}$ and we will use $\loc^*(\mathcal{A})$ to denote the smallest localising $\sfT$-submodule of $\K$ containing $\mathcal{A}$. Extending our earlier notation we let $\Loc_\sfT^*(\K)$ denote the lattice of localising submodules of $\K$ with respect to the action of $\sfT$. Usually the action in question is clear and we omit the $\sfT$ from the notation.

Given in addition a class $\mathcal{X}$ of objects of $\sfT$ we denote by
\begin{displaymath}
\mathcal{X}*\mathcal{A} = \loc^*( X*A\; \vert \; X\in \mathcal{X}, A\in \mathcal{A})
\end{displaymath}
the localising submodule generated by products, with respect to the action, of the objects from $\mathcal{X}$ and $\mathcal{A}$.
\end{notation}

We have the following useful technical result for working with tensor-ideals and submodules generated by prescribed sets of objects.

\begin{lem}[\cite{StevensonActions}*{Lemma~3.12}]\label{lem_loc_commutes2}
Formation of localising $\sfT$-submodules commutes with the action, i.e.\ given a collection of objects $\mathcal{X}$ of $\sfT$ and a collection of objects $\mathcal{A}$ of $\K$ we have
\begin{align*}
\loc^\otimes( \mathcal{X}) * \loc( \mathcal{A}) &= \loc( \mathcal{X})* \loc(\mathcal{A}) \\
&= \mathcal{X} * \mathcal{A}\\
&= \loc( Z*(X*A) \; \vert \; Z\in \sfT, X\in \mathcal{X}, A\in \mathcal{A}).
\end{align*}
\end{lem}

Another useful lemma in examples, which frequently removes the need to keep track of submodules versus localising subcategories, is the following analogue of Lemmas \ref{lem_small_unit_gen} and \ref{lem_big_unit_gen}.

\begin{lem}[\cite{StevensonActions}*{Lemma~3.13}]\label{lem_gen_locpreserving}
If $\sfT$ is generated as a localising subcategory by the tensor unit $\mathbf{1}$ then every localising subcategory of $\K$ is a $\sfT$-submodule.
\end{lem}

\subsection{Supports}

We now introduce supports for $\K$ relative to the action of $\sfT$. This is done in the naive way, extending what we have seen in the last section. Recall from Definition~\ref{defn_gen_ptfunctors2} that for each $x\in \Spc\sfT^c$ we have objects
\begin{displaymath}
\mathit{\Gamma}_x\mathbf{1} = \mathit{\Gamma}_{\mathcal{V}(x)}\mathbf{1} \otimes L_{\mathcal{Z}(x)}\mathbf{1}.
\end{displaymath}

\begin{notation}
We use $\mathit{\Gamma}_x\K$, for $x\in \Spc \sfT^c$, to denote the essential image of $\mathit{\Gamma}_x\mathbf{1}*(-)$. It is a $\sfT$-submodule as for any $X \in \sfT$ and $A\in \mathit{\Gamma}_x \K$
\begin{align*}
X*A \iso X*(\mathit{\Gamma}_x\mathbf{1}*A') \iso \mathit{\Gamma}_x\mathbf{1}*(X*A')
\end{align*}
for some $A'\in \K$ as, by virtue of being in the essential image of $\mathit{\Gamma}_x\mathbf{1}\otimes(-)$, we must have such an $A'$ and an isomorphism
\begin{displaymath}
A\cong \mathit{\Gamma}_x\mathbf{1}\otimes A'.
\end{displaymath}
Similarly we will often write $\mathit{\Gamma}_xA$ as shorthand for $\mathit{\Gamma}_x\mathbf{1}*A$.
\end{notation}

\begin{defn}\label{defn_actionsupport}
Given $A$ in $\K$ we define the support of $A$ to be the set
\begin{displaymath}
\supp_{(\sfT,*)} A = \{x\in \Spc \sfT^c \; \vert \; \mathit{\Gamma}_x A \neq 0\}.
\end{displaymath}
When the action in question is clear we will omit the subscript from the notation.
\end{defn}

We have analogous properties for this notion of support as for the support when $\sfT$ acts on itself as in Proposition~\ref{prop_supp_prop}.

\begin{prop}\label{prop_abs_supp_prop}
The support assignment $\supp_{(\sfT,*)}$ satisfies the following properties:
\begin{itemize}
\item[$(a)$] given a triangle
\begin{displaymath}
A \to B \to C \to \S A
\end{displaymath}
in $\K$ we have $\supp B \cie \supp A \un \supp C$;
\item[$(b)$] for any $A$ in $\K$ and $i\in \ZZ$
\begin{displaymath}
\supp A = \supp \S^i A;
\end{displaymath}
\item[$(c)$] given a set-indexed family $\{A_\l\}_{\l \in \Lambda}$ of objects of $\K$ there is an equality
\begin{displaymath}
\supp \coprod_\l A_\l = \Un_\l \supp A_\l;
\end{displaymath}
\item[$(d)$] the support satisfies the separation axiom, i.e.\ for every specialization closed subset $\mathcal{V} \cie \Spc \sfT^c$ and every object $A$ of $\K$
\begin{align*}
&\supp \mathit{\Gamma}_\mathcal{V}\mathbf{1}* A = (\supp A) \intersec \mathcal{V} \\
&\supp L_\mathcal{V}\mathbf{1}*A = (\supp A) \intersec (\Spc \sfT^c \setminus \mathcal{V}).
\end{align*}
\end{itemize}
\end{prop}

\subsection{The local-to-global principle and parametrising submodules}

As we saw an action allows us to define a sublattice $\Loc^*(\K)$ of the collection of all localising subcategories of $\K$. The hope is twofold: on one hand we hope that the lattice $\Loc^*(\K)$ is more tractable than $\Loc(\K)$ and on the other hand we hope the action gives us the necessary tools to analyse $\Loc^*(\K)$. In this section we will define assignments relating supports to localising submodules. These assignments are a step toward achieving our two hopes. We then introduce the local-to-global principle, a version of which was first formalised in \cite{BIKstrat2} but is already implicit in the work of Neeman \cite{NeeChro}, and explain some of its consequences for these assignments. We will prove a general result showing that the local-to-global principle holds in many cases of interest (see also Aside~\ref{aside_ltg} for a brief discussion of an even more general approach). This powerful technical result allows us to reduce classification problems into smaller (but usually still difficult) pieces. Along the way we will finally show that the support, in the sense we have defined above, is sufficiently refined to determine whether or not an object is $0$.

\begin{defn}\label{defn_vis_sigmatau}
We define order preserving assignments
\begin{displaymath}
\left\{ \begin{array}{c}
\text{subsets of}\; \Spc\sfT^c 
\end{array} \right\}
\xymatrix{ \ar[r]<1ex>^\t \ar@{<-}[r]<-1ex>_\s &} 
\Loc^*(\K)
\end{displaymath}
where both collections are ordered by inclusion. For a localising submodule $\sfL$ we set
\begin{displaymath}
\s(\sfL) = \supp \sfL = \{x \in \Spc\sfT^c \; \vert \; \mathit{\Gamma}_x\sfL \neq 0\}
\end{displaymath}
and for a subset $W\cie \Spc\sfT^c$ we set
\begin{displaymath}
\t(W) = \{A \in \K \; \vert \; \supp A \cie W\}.
\end{displaymath}
Both of these are well defined; this is clear for $\s$ and for $\t$ it follows from Proposition \ref{prop_abs_supp_prop}.
\end{defn}

\begin{defn}\label{defn_ltg}
We say $\sfT\times\K \stackrel{*}{\to} \K$ satisfies the \emph{local-to-global principle} if for each $A$ in $\K$
\begin{displaymath}
\loc^*(A) = \loc^*(\mathit{\Gamma}_x A \; \vert \; x\in \Spc \sfT^c).
\end{displaymath}
\end{defn}

The local-to-global principle has the following rather pleasing consequences for the assignments $\s$ and $\t$ of Definition \ref{defn_vis_sigmatau}.

\begin{lem}\label{general_im_tau}
Suppose the local-to-global principle holds for the action of $\sfT$ on $\K$ and let $W$ be a subset of $\Spc \sfT^c$. Then
\begin{displaymath}
\t(W) = \loc^*(\mathit{\Gamma}_x\K \; \vert \; x\in W\intersec \s\K ).
\end{displaymath}
\end{lem}
\begin{proof}
By the local-to-global principle we have for every object $A$ of $\K$ an equality
\begin{displaymath}
\loc^*(A) = \loc^*(\mathit{\Gamma}_xA \; \vert \; x\in \Spc\sfT^c).
\end{displaymath}
Thus
\begin{align*}
\t(W) &= \loc^*(A \; \vert \; \supp A \cie W) \\
&= \loc^*(\mathit{\Gamma}_x A \; \vert \; A\in \K, x\in W) \\
&= \loc^*(\mathit{\Gamma}_x A \; \vert \; A\in \K, x\in W\intersec \s\K) \\
&= \loc^*( \mathit{\Gamma}_x\K \; \vert \; x\in W\intersec \s\K ).
\end{align*}
\end{proof}

\begin{prop}\label{general_tau_inj}
Suppose the local-to-global principle holds for the action of $\sfT$ on $\K$ and let $W$ be a subset of $\Spc \sfT^c$. Then there is an equality of subsets
\begin{displaymath}
\s\t(W) = W \intersec \s\K.
\end{displaymath}
In particular, $\t$ is injective when restricted to subsets of $\s\K$.
\end{prop}
\begin{proof}
With $W\cie \Spc \sfT^c$ as in the statement we have
\begin{align*}
\s\t(W) &= \supp \t(W) \\
&= \supp \loc^*( \mathit{\Gamma}_x\K \; \vert \; x\in W\intersec \s\K ),
\end{align*}
the first equality by definition and the second by the last lemma. Thus $\s\t(W) = W\intersec \s\K$ as claimed: by the properties of the support (Proposition \ref{prop_abs_supp_prop}) we have $\s\t(W) \cie W\intersec \s\K$ and it must in fact be all of $W\intersec \s\K$ as $x\in \s\K$ if and only if $\mathit{\Gamma}_x\K$ contains a non-zero object.
\end{proof}

Now we go about showing that we have access to these results in significant generality.

\begin{defn}\label{defn_hocolim}
We will say the tensor triangulated category $\sfT$ \emph{has a model} if it occurs as the homotopy category of a monoidal model category.
\end{defn}

Our main interest in such categories is that the existence of a monoidal model provides a good theory of homotopy colimits compatible with the tensor product. 

\begin{rem}
Of course instead of requiring that $\sfT$ arose from a monoidal model category we could, for instance, ask that $\sfT$ was the underlying category of a stable monoidal derivator. In fact we will only use directed homotopy colimits so one could use a weaker notion of a stable monoidal ``derivator'' only having homotopy left and right Kan extensions for certain diagrams; to be slightly more precise one could just ask for homotopy left and right Kan extensions along the smallest full $2$-subcategory of the category of small categories satisfying certain natural closure conditions and containing the ordinals (one can see the discussion before \cite{GrothPSD} Definition 4.21 for further details).
\end{rem}

We begin by showing that, when $\sfT$ has a model, taking the union of a chain of specialisation closed subsets is compatible with taking the homotopy colimit of the associated idempotents.

\begin{lem}\label{lem_hocolim_ltg}
Suppose $\sfT$ has a model. Then for any chain $\{\mathcal{V}_i\}_{i\in I}$ of specialisation closed subsets of $\Spc \sfT^c$ with union $\mathcal{V}$ there is an isomorphism
\begin{displaymath}
\mathit{\Gamma}_\mathcal{V}\mathbf{1} \iso \hocolim \mathit{\Gamma}_{\mathcal{V}_i}\mathbf{1}
\end{displaymath}
where the structure maps are the canonical ones.
\end{lem}
\begin{proof}
Each $\mathcal{V}_i$ is contained in $\mathcal{V}$ so there are corresponding inclusions for $i<j$ 
\begin{displaymath}
\mathit{\Gamma}_{\mathcal{V}_i}\sfT \cie \mathit{\Gamma}_{\mathcal{V}_j}\sfT \cie \mathit{\Gamma}_{\mathcal{V}}\sfT.
\end{displaymath}
These inclusions induce canonical morphisms between the corresponding tensor-idempotents which fit into a commutative triangle
\begin{displaymath}
\xymatrix{
\mathit{\Gamma}_{\mathcal{V}_i}\mathbf{1} \ar[rr] \ar[dr] && \mathit{\Gamma}_{\mathcal{V}}\mathbf{1} \\
& \mathit{\Gamma}_{\mathcal{V}_j}\mathbf{1} \ar[ur]
}
\end{displaymath}
We remind the reader that $\mathit{\Gamma}_{\mathcal{V}_i}\sfT$ denotes the localising tensor-ideal generated by those compact objects with support in $\mathcal{V}_i$ and $\mathit{\Gamma}_{\mathcal{V}_i}\mathbf{1}$ is the associated idempotent for the acyclisation (see the discussion before Definition~\ref{defn_gen_ptfunctors}).

Commutativity of these triangles gives us an induced morphism from the homotopy colimit of the $\mathit{\Gamma}_{\mathcal{V}_i}\mathbf{1}$ to $\mathit{\Gamma}_\mathcal{V}\mathbf{1}$ which we complete to a triangle
\begin{displaymath}
\hocolim_I \mathit{\Gamma}_{\mathcal{V}_i}\mathbf{1} \to \mathit{\Gamma}_\mathcal{V}\mathbf{1} \to Z \to \S \hocolim_I \mathit{\Gamma}_{\mathcal{V}_i}\mathbf{1}.
\end{displaymath}
In order to prove the lemma it is sufficient to show that $Z$ is isomorphic to the zero object in $\sfT$. If this is the case then the first map in the above triangle is the desired isomorphism.

The argument given in \cite{BousfieldBAS} can be extended to prove that localising subcategories are closed under homotopy colimits and so this triangle consists of objects of $\mathit{\Gamma}_\mathcal{V}\sfT$. By definition $\mathit{\Gamma}_\mathcal{V}\sfT$ is the full subcategory of $\sfT$ generated by those objects of $\sfT^c$ whose support is contained in $\mathcal{V}$. Thus $Z\iso 0$ if and only if for each compact object $t$ with $\supp t \cie \mathcal{V}$ we have $\Hom(t,Z) = 0$. We note there is no ambiguity here as by Proposition~\ref{prop_supp_prop} the two notions of support, that of \cite{BaSpec} and \cite{BaRickard}, agree for compact objects. In particular, the support of any compact object is closed.

Recall from Theorem~\ref{thm_spectral} that $\Spc \sfT^c$ is spectral (see Definition~\ref{defn_spectral}), and we have assumed it is also noetherian. Thus the closed subset $\supp t$ can be written as a finite union of irreducible closed subsets. We can certainly find a $j\in I$ so that $\mathcal{V}_j$ contains the generic points of these finitely many irreducible components which implies $\supp t \cie \mathcal{V}_j$ by specialisation closure. Therefore, by adjunction, it is enough to show 
\begin{align*}
\Hom(t, Z) &\iso \Hom(\mathit{\Gamma}_{\mathcal{V}_j} t ,Z) \\
&\iso \Hom(t, \mathit{\Gamma}_{\mathcal{V}_j} Z)
\end{align*}
is zero, as this implies $Z \iso 0$ and we get the claimed isomorphism.

In order to show the claimed hom-set vanishes let us demonstrate that $\mathit{\Gamma}_{\mathcal{V}_j}Z$ is zero. Observe that tensoring the structure morphisms $\mathit{\Gamma}_{\mathcal{V}_{i_1}}\mathbf{1} \to \mathit{\Gamma}_{\mathcal{V}_{i_2}}\mathbf{1}$ for $i_2\geq i_1\geq j$ with $\mathit{\Gamma}_{\mathcal{V}_j}\mathbf{1}$ yields canonical isomorphisms
\begin{displaymath}
\mathit{\Gamma}_{\mathcal{V}_j}\mathbf{1} \iso \mathit{\Gamma}_{\mathcal{V}_j}\mathbf{1}\otimes \mathit{\Gamma}_{\mathcal{V}_{i_1}}\mathbf{1} \stackrel{\sim}{\to} \mathit{\Gamma}_{\mathcal{V}_j}\mathbf{1}\otimes \mathit{\Gamma}_{\mathcal{V}_{i_2}}\mathbf{1} \iso \mathit{\Gamma}_{\mathcal{V}_j}\mathbf{1}.
\end{displaymath}
Thus applying $\mathit{\Gamma}_{\mathcal{V}_j}$ to the sequence $\{\mathit{\Gamma}_{\mathcal{V}_i}\mathbf{1}\}_{i\in I}$ gives a diagram whose homotopy colimit is $\mathit{\Gamma}_{\mathcal{V}_j}\mathbf{1}$. From this we deduce that the first morphism in the resulting triangle
\begin{displaymath}
\mathit{\Gamma}_{\mathcal{V}_j}\mathbf{1} \otimes \hocolim_I \mathit{\Gamma}_{\mathcal{V}_i}\mathbf{1} \to \mathit{\Gamma}_{\mathcal{V}_j}\mathbf{1} \to \mathit{\Gamma}_{\mathcal{V}_j}Z 
\end{displaymath}
is an isomorphism; since $\sfT$ is the homotopy category of a monoidal model category the tensor product commutes with homotopy colimits. This forces $\mathit{\Gamma}_{\mathcal{V}_j}Z \iso 0$ completing the proof.
\end{proof}

\begin{lem}\label{lem_Amnon2.10_uber}
Let $P\cie \Spc \sfT^c$ be given and suppose $A$ is an object of $\K$ such that $\mathit{\Gamma}_x A \iso 0$ for all $x\in (\Spc \sfT^c \setminus P)$. If $\sfT$ has a model then $A$ is an object of the localising subcategory
\begin{displaymath}
\sfL = \loc( \mathit{\Gamma}_y\K \; \vert \; y\in P).
\end{displaymath}
\end{lem}
\begin{proof}
Denote by $\mathcal{P}(\Spc \sfT^c)$ the power set of $\Spc\sfT^c$ and let $\Lambda \cie \mathcal{P}(\Spc \sfT^c)$ be the set of specialisation closed subsets $\mathcal{W}$ such that $\mathit{\Gamma}_\mathcal{W}A$ is in $\sfL = \loc( \mathit{\Gamma}_y\K \; \vert \; y\in P)$. We first note that $\Lambda$ is not empty. Indeed, as $\sfT^c$ is rigid the only compact objects with empty support are the zero objects by \cite{BaFilt}*{Corollary~2.5} so
\begin{displaymath}
\mathit{\Gamma}_\varnothing\sfT = \loc(t\in \sfT^c \; \vert \; \supp_{(\sfT,\otimes)} t = \varnothing ) = \loc( 0 )
\end{displaymath}
giving $\mathit{\Gamma}_\varnothing A = 0$ and hence $\varnothing \in \Lambda$. 

Since $\sfL$ is localising, Lemma \ref{lem_hocolim_ltg} shows the set $\Lambda$ is closed under taking increasing unions: as mentioned above the argument in \cite{BousfieldBAS} extends to show that localising subcategories are closed under directed homotopy colimits in our situation. Thus $\Lambda$ contains a maximal element $Y$ by Zorn's lemma. We will show that $Y = \Spc \sfT^c$.

Suppose $Y\neq \Spc \sfT^c$. Since $\Spc \sfT^c$ is noetherian the subspace $\Spc\sfT^c \setminus Y$ contains an element $z$ maximal with respect to specialisation. We have 
\begin{displaymath}
L_Y\mathbf{1}\otimes\mathit{\Gamma}_{Y\un \{z\}}\mathbf{1} \iso \mathit{\Gamma}_z\mathbf{1}
\end{displaymath}
as $Y\un \{z\}$ is specialisation closed by maximality of $z$ and Remark~\ref{rem_bik6.2_uber} tells us that we can use any suitable pair of Thomason subsets to define $\mathit{\Gamma}_z\mathbf{1}$. So $L_Y\mathit{\Gamma}_{Y\un \{z\}}A \iso \mathit{\Gamma}_z A$ and by our hypothesis on vanishing either $\mathit{\Gamma}_z \K \cie \sfL$ or \mbox{$\mathit{\Gamma}_z A = 0$}. Consider the triangle
\begin{displaymath}
\xymatrix{
\mathit{\Gamma}_Y\mathit{\Gamma}_{Y\un \{z\}} A \ar[d]_{\wr} \ar[r] & \mathit{\Gamma}_{Y\un \{z\}}A \ar[r] & L_Y\mathit{\Gamma}_{Y\un \{z\}}A \ar[d]^{\wr} \\
\mathit{\Gamma}_Y A & & \mathit{\Gamma}_z A
}
\end{displaymath}
We see that in either case, since $\mathit{\Gamma}_YA$ is in $\sfL$ and $\sfL$ is localising, the object $\mathit{\Gamma}_{Y\cup \{z\}}A$ is in $\sfL$. Thus $Y\un \{z\}\in \Lambda$ which contradicts the maximality of $Y$. Hence $Y = \Spc \sfT^c$ and so $A$ is in $\sfL$.
\end{proof}

\begin{prop}\label{prop_ltg_idempotents}
Suppose $\sfT$ has a model. Then the local-to-global principle holds for the action of $\sfT$ on $\K$. Explicitly, for any $A$ in $\K$ there is an equality of $\sfT$-submodules
\begin{displaymath}
\loc^*( A ) = \loc^*( \mathit{\Gamma}_x A \; \vert \; x\in \supp A ).
\end{displaymath}
\end{prop}
\begin{proof}
By Lemma \ref{lem_Amnon2.10_uber} applied to the action
\begin{displaymath}
\sfT\times \sfT \stackrel{\otimes}{\to} \sfT
\end{displaymath}
we see $\sfT = \loc( \mathit{\Gamma}_x\sfT\; \vert \; x\in \Spc \sfT^c)$. Since $\mathit{\Gamma}_x\sfT = \loc^\otimes( \mathit{\Gamma}_x\mathbf{1})$ it follows that the set of objects $\{\mathit{\Gamma}_x\mathbf{1} \; \vert \; x\in \Spc \sfT^c \}$ generates $\sfT$ as a localising tensor-ideal. By Lemma \ref{lem_loc_commutes2} given an object $A\in \K$ we get a generating set for  $\sfT*\loc( A)$:
\begin{displaymath}
\sfT* \loc( A ) = \loc^\otimes( \mathit{\Gamma}_x\mathbf{1} \; \vert \; x\in \Spc \sfT) * \loc( A )= \loc^*( \mathit{\Gamma}_x A \; \vert \; x\in \supp A).
\end{displaymath}
But it is also clear that $\sfT = \loc^\otimes( \mathbf{1} )$ so, by Lemma \ref{lem_loc_commutes2} again,
\begin{displaymath}
\sfT* \loc( A ) = \loc^\otimes( \mathbf{1} ) * \loc( A ) = \loc^*( A )
\end{displaymath}
and combining this with the other string of equalities gives
\begin{displaymath}
\loc^*( A ) = \sfT*\loc( A ) = \loc^*( \mathit{\Gamma}_x A \; \vert \; x\in \supp A)
\end{displaymath}
which completes the proof.
\end{proof}

We thus have the following theorem concerning the local-to-global principle for actions of rigidly-compactly generated tensor triangulated categories.

\begin{thm}\label{thm_general_ltg}
Suppose $\sfT$ is a rigidly-compactly generated tensor triangulated category with a model and such that $\Spc \sfT^c$ is noetherian. Then for any compactly generated triangulated category $\K$, for instance $\K = \sfT$, we have:
\begin{itemize}
\item[$(i)$] The local-to-global principle holds for any action of $\sfT$ on $\K$;
\item[$(ii)$] The associated support theory detects vanishing of objects, i.e.\ $A \in \K$ is zero if and only if $\supp X = \varnothing$;
\item[$(iii)$] For any chain $\{\mathcal{V}_i\}_{i\in I}$ of specialisation closed subsets of $\Spc \sfT^c$ with union $\mathcal{V}$ there is an isomorphism
\begin{displaymath}
\mathit{\Gamma}_\mathcal{V}\mathbf{1} \iso \hocolim \mathit{\Gamma}_{\mathcal{V}_i}\mathbf{1}
\end{displaymath}
where the structure maps are the canonical ones.
\end{itemize}
\end{thm}
\begin{proof}
That (iii) always holds is the content of Lemma \ref{lem_hocolim_ltg} and we have proved in Proposition \ref{prop_ltg_idempotents} that (i) holds.
To see (i) implies (ii), for an action $\sfT *\K \to \K$, observe that if $\supp X = \varnothing$ for an object $X$ of $\K$ then the local-to-global principle yields
\begin{displaymath}
\loc^*( X ) = \loc^*( \mathit{\Gamma}_x X \; \vert \; x\in \Spc\sfT^c ) = \loc^*( 0 )
\end{displaymath}
so $X\iso 0$.
\end{proof}

\begin{aside}\label{aside_ltg}
There are other situations in which the local-to-global principle is known to hold\textemdash{}one is not forced to only consider $\sfT$ whose compacts have noetherian spectrum. A general criterion for verifying the local-to-global principle is given in \cite{StevensonLTG} and it is shown that there are other settings in which one has access to this tool. However, the local-to-global principle does not come for free. In \cite{StevensonAbsFlat} examples are given where the local-to-global principle fails.
\end{aside}

\subsection{An example: commutative noetherian rings}\label{sec_Neeman_ex}
As an example we sketch a proof of Neeman's classification, as proved in \cite{NeeChro}, of the localising subcategories of $\D(R)$ for a commutative noetherian ring $R$. Our goal is to illustrate how one applies the machinery thusfar developed in a concrete setting. The details of the argument are not so different from those in Amnon's paper, although the abstract machinery we have developed takes care of many of the technical points. 

Let us begin by recalling the theorem.

\begin{thm}[\cite{NeeChro}*{Theorem~2.8}]
Let $R$ be a commutative noetherian ring. There is an isomorphism of lattices
\begin{displaymath}
\left\{ \begin{array}{c}
\text{subsets of}\; \Spec R 
\end{array} \right\}
\xymatrix{ \ar[r]<1ex>^\t \ar@{<-}[r]<-1ex>_\s &} 
\Loc \D(R)
\end{displaymath}
\end{thm}

As we have claimed in Example~\ref{ex_ringspec} there is a homeomorphism $\Spc \D^\mathrm{perf}(R) \cong \Spec R$ and we regard this as an identification. Let us now recall what the various relevant tensor-idempotents are in this situation. We introduced, in Example~\ref{ex_ring_ex}, the stable Koszul complex 
\begin{displaymath}
K_\infty(f) = \cdots \to 0 \to R \to R_f \to 0 \to \cdots
\end{displaymath}
and the analogue for any ideal $I$ in $R$. In this example we also noted that there are isomorphisms in $\D(R)$ for an ideal $I$ and a prime ideal $\mathfrak{p}$
\begin{itemize}
\item[(1)] $\mathit{\Gamma}_{\mathcal{V}(I)}R \iso K_{\infty}(I)$;
\item[(2)] $L_{\mathcal{V}(I)}R \iso \check{C}(I)$;
\item[(3)] $L_{\mathcal{Z}(\mathfrak{p})}R \iso R_\mathfrak{p}$.
\end{itemize}
In particular, $\mathit{\Gamma}_\mathfrak{p}R \cong K_\infty(\mathfrak{p})_\mathfrak{p}$ for each prime ideal $\mathfrak{p}\in \Spec R$.

The unbounded derived category $\D(R)$ is certainly rigidly-compactly generated and we have assumed $R$ is noetherian so Theorem~\ref{thm_general_ltg} guarantees for us that the local-to-global principle holds. Thus, considering the assignments
\begin{displaymath}
\left\{ \begin{array}{c}
\text{subsets of}\; \Spec R 
\end{array} \right\}
\xymatrix{ \ar[r]<1ex>^\t \ar@{<-}[r]<-1ex>_\s &} 
\Loc\D(R)
\end{displaymath}
given by
\begin{displaymath}
\s(\sfL) = \supp \sfL = \{\mathfrak{p} \in \Spec R \; \vert \; \mathit{\Gamma}_\mathfrak{p}\sfL \neq 0\}
\end{displaymath}
for a localising subcategory $\sfL$, and
\begin{displaymath}
\t(W) = \{A \in \D(R) \; \vert \; \supp A \cie W\}
\end{displaymath}
for a subset $W$ of $\Spec R$, we know by Proposition~\ref{general_tau_inj} that $\tau$ is a split monomorphism when restricted to $\s\D(R)$. One sees easily, using the explicit description, that none of the $\mathit{\Gamma}_\mathfrak{p}R$ are zero and so, in fact, $\tau$ is just injective.

In order to show $\tau$ is an isomorphism it is enough to check that each $\mathit{\Gamma}_\mathfrak{p}\D(R)$ is \emph{minimal} in the sense that it has no proper non-trivial localising subcategories. Minimality of $\mathit{\Gamma}_\mathfrak{p}\D(R)$ is equivalent to the statement that for any non-zero $A\in \mathit{\Gamma}_\mathfrak{p}\D(R)$ we have
\begin{displaymath}
\loc(A) = \mathit{\Gamma}_\mathfrak{p}\D(R).
\end{displaymath}
Suppose this is the case. Then for $\sfL\in \Loc\D(R)$ we clearly have a containment
\begin{displaymath}
\sfL \subseteq \t\s\sfL.
\end{displaymath}
On the other hand, if $A \in \t\s\sfL$ then we know $\supp A \subseteq \s\sfL$ and so by the local-to-global principle
\begin{displaymath}
A\in \loc(A) = \loc(\mathit{\Gamma}_\mathfrak{p}A \; \vert \; \mathfrak{p}\in \s\sfL).
\end{displaymath}
If $\mathfrak{p}\in \s\sfL$ then $\mathit{\Gamma}_\mathfrak{p}\sfL \neq 0$ and so, by our minimality assumption, $\mathit{\Gamma}_\mathfrak{p}\D(R) \subseteq \sfL$. This shows $A \in \sfL$ proving $\sfL = \t\s\sfL$.

We now sketch the proof that the $\mathit{\Gamma}_\mathfrak{p}\D(R)$ are minimal. For each prime ideal $\mathfrak{p}$ we will denote by $k(\mathfrak{p}) = R_\mathfrak{p}/\mathfrak{p}R_\mathfrak{p}$ the associated residue field.

\begin{lem}\label{lem_kpyo}
There are equalities of localising subcategories
\begin{displaymath}
\mathit{\Gamma}_\mathfrak{p}\D(R) = \loc(\mathit{\Gamma}_\mathfrak{p}R) = \loc(k(\mathfrak{p})).
\end{displaymath}
\end{lem}
\begin{proof}
The first equality follows from the fact that, by Lemma~\ref{lem_big_unit_gen}, every localising subcategory of $\D(R)$ is a tensor-ideal. In order to prove the second equality, first note that
\begin{displaymath}
\mathit{\Gamma}_\mathfrak{p}R \otimes k(\mathfrak{p}) \cong k(\mathfrak{p})
\end{displaymath}
and so $k(\mathfrak{p}) \in \mathit{\Gamma}_\mathfrak{p}\D(R)$. To prove the reverse containment, we begin by noting that $K(\mathfrak{p}^i)_\mathfrak{p}$, the usual Koszul complex for $\mathfrak{p}^i$ localised at $\mathfrak{p}$ is contained in $\thick(k(\mathfrak{p}))$. This is a straightforward consequence of the fact that over $R_\mathfrak{p}$ the complex $K(\mathfrak{p}^i)_\mathfrak{p}$ has finite length homology. One can write the localised stable Koszul complex $K_\infty(\mathfrak{p})_\mathfrak{p}$ as a homotopy colimit of the $K(\mathfrak{p}^i)_\mathfrak{p}$ and so $K_\infty(\mathfrak{p})_\mathfrak{p} \in \loc(k(\mathfrak{p}))$. This completes the proof of the second equality. 
\end{proof}

\begin{rem}
Using this lemma one can check that the support we have defined for $\D(R)$ agrees with the support of Neeman which is defined in terms of the residue fields.
\end{rem}

\begin{lem}
The localising subcategory $\mathit{\Gamma}_\mathfrak{p}\D(R)$ is minimal.
\end{lem}
\begin{proof}
If $A\in \D(R)$ then one can check, by factoring the functor $k(\mathfrak{p})\otimes-$ through $\D(k(\mathfrak{p}))$, that 
\begin{displaymath}
k(\mathfrak{p})\otimes A \cong \bigoplus_{i\in \ZZ}\limits \Sigma^i k(\mathfrak{p})^{(\lambda_i)}
\end{displaymath}
for some set of cardinals $\lambda_i$, where $k(\mathfrak{p})^{(\lambda_i)}$ denotes the coproduct of $\lambda_i$ copies of $k(\mathfrak{p})$. One then proceeds by arguing that if $A\neq 0$ lies in $\mathit{\Gamma}_\mathfrak{p}\D(R)$ we must have $k(\mathfrak{p})\otimes A\neq 0$ and so $\loc(A)$ contains $k(\mathfrak{p})$. Hence
\begin{displaymath}
\loc(k(\mathfrak{p})) \subseteq \loc(A) \subseteq \mathit{\Gamma}_\mathfrak{p}\D(R),
\end{displaymath}
which shows $\loc(A) = \mathit{\Gamma}_\mathfrak{p}\D(R)$.
\end{proof}

\section{Further applications and examples}

In this final section we concentrate on giving some further examples to illustrate how one can use the local-to-global principle in practice.

\subsection{Singularity categories of hypersurfaces}

Here we outline the solution to the classification problem for localising subcategories of singularity categories of hypersurface rings. The intention is to indicate the manner in which these computations usually proceed. In particular, we further highlight the manner in which the local-to-global principle is used. We begin by defining the main players and explaining the action we will use. Further details, including the analogous results in the non-affine case, can be found in \cite{Stevensonclass}.

We can, for the time being, work quite generally. Suppose $R$ is an arbitrary commutative noetherian ring. Then one defines a category
\begin{displaymath}
\D_{\mathrm{Sg}}(R) := \D^\mathrm{b}(\modu R) / \D^{\mathrm{perf}}(R),
\end{displaymath}
where $\D^\mathrm{b}(\modu R)$ is the bounded derived category of finitely generated $R$-modules and $\D^\mathrm{perf}(R)$ is the full subcategory of complexes locally isomorphic to bounded complexes of finitely generated projectives. This category measures the singularities of $R$. In particular, $\D_{\mathrm{Sg}}(R)$ vanishes if and only if $R$ is regular, it is related to other measures of the singularities of $R$ for example maximal Cohen-Macaulay modules (see \cite{Buchweitzunpub}), and its properties reflect the severity of the singularities of $R$. The particular category which will concern us is the stable derived category of Krause \cite{KrStab}, namely
\begin{displaymath}
\sfS(R) := \K_\mathrm{ac}(\Inj R)
\end{displaymath}
the homotopy category of acyclic complexes of injective $R$-modules. We slightly abuse standard terminology by calling $\sfS(R)$ the \emph{singularity category} of $R$. We will also need to consider the following categories
\begin{displaymath}
\D(R) := \D(\Modu R), \quad \text{and} \quad \K(R) := K(\Modu R)
\end{displaymath}
the unbounded derived category of $R$ and the homotopy category of $R$-modules. As indicated above we denote by $\Inj R$ the category of injective $R$-modules and we write $\Flat R$ for the category of flat $R$-modules.

The main facts about $\sfS(R)$ that we will need are summarised in the following theorem of Krause.

\begin{thm}[\cite{KrStab} Theorem 1.1]\label{thm_recol}
Let $R$ be a commutative noetherian ring.
\begin{itemize}
\item[$(1)$] There is a recollement 
\begin{displaymath}
\xymatrix{
\sfS(R) \ar[rr]^{I} && \K(\Inj R) \ar[rr]^{Q} \ar@/^1pc/[ll]^{I_\r} \ar@/_1pc/[ll]_{I_{\l}} &&  \D(R) \ar@/^1pc/[ll]^{Q_\r} \ar@/_1pc/[ll]_{Q_{\l}}
}
\end{displaymath}
where each functor is right adjoint to the one above it. 
\item[$(2)$] The triangulated category $\K(\Inj R)$ is compactly generated, and $Q$ induces an equivalence
\begin{displaymath}
\K(\Inj R)^{\mathrm{c}} \to \D^b(\modu R).
\end{displaymath}
\item[$(3)$] The sequence
\begin{displaymath}
\xymatrix{
\D(R) \ar[r]^(0.5){Q_\l} & \K(\Inj R) \ar[r]^(0.5){I_{\l}} & \sfS(R)
}
\end{displaymath}
is a localisation sequence. Therefore $\sfS(R)$ is compactly generated, and the composite $I_\l \circ Q_\r$ induces (up to direct factors) an equivalence
\begin{displaymath}
\D_{\mathrm{Sg}}(R) \to \sfS(R)^{\mathrm{c}}.
\end{displaymath}
\end{itemize}
\end{thm}

Now let us outline the proof that there is an action
\begin{displaymath}
\D(R) \times \sfS(R) \stackrel{\odot}{\to} \sfS(R)
\end{displaymath}
in the sense of Definition~\ref{defn_action}. This comes down to showing one can take K-flat resolutions (see Definition~\ref{defn_Kflat}) of objects of $\D(R)$ in a way which, although not necessarily initially functorial, becomes functorial after tensoring with an acyclic complex of injectives. Our starting point is work of Neeman \cite{NeeFlat} and Murfet \cite{MurfetThesis} which provides a category that naturally acts on $\K(\Inj R)$. As we work here in the affine case Neeman's results are sufficient for our purposes. However, we provide references to the work of Murfet as this makes the results necessary to generalise to the non-affine case apparent.

There is an obvious action of the category $\K(\Flat R)$ on $\K(\Inj R)$. It is given by simply taking the tensor product of complexes. This does in fact define an action as, since $R$ is noetherian, the tensor product of a complex of flats with a complex of injectives is again a complex of injectives. However, $\K(\Flat R)$ may contain many objects which act trivially. Thus we will now pass to a more manageable quotient.

In order to do so we need the notion of pure acyclicity. In \cite{MurfetTAC} a complex $F$ in $\K(\Flat R)$ is defined to be \emph{pure acyclic} if it is exact and has flat syzygies. Such complexes form a triangulated subcategory of $\K(\Flat R)$ which we denote by $\K_{\mathrm{pac}}(\Flat R)$ and we say that a morphism with pure acyclic mapping cone is a \emph{pure quasi-isomorphism}.  The point is that tensoring a pure acyclic complex of flats with a complex of injectives yields a contractible complex by \cite{NeeFlat}*{Corollary 9.7}.

Following Neeman and Murfet we consider $\sfN(\Flat R)$, which is defined to be the quotient $\K(\Flat R)/\K_{\mathrm{pac}}(\Flat R)$. What we have observed above shows the action of $\K(\Flat R)$ on $\K(\Inj R)$ factors via the projection $\K(\Flat R) \to \sfN(\Flat R)$. Next we describe a suitable subcategory of $\sfN(\Flat R)$ which will act on $\sfS(R)$.

Recall from above that $\sfS(R)$ is a compactly generated triangulated category. Consider $E = \coprod_\l E_\l$ where $E_\l$ runs through a set of representatives for the isomorphism classes of compact objects in $\sfS(R)$. We define a homological functor $H\colon \K(\Flat R) \to \mathrm{Ab}$ by setting, for $F$ an object of $\K(\Flat R)$,
\begin{displaymath}
H(F) = H^0(F\otimes_{R} E)
\end{displaymath}
where the tensor product is taken in $\K(R)$. This is a coproduct preserving homological functor since we are merely composing the exact coproduct preserving functor $(-)\otimes_{R} E$ with the coproduct preserving homological functor $H^0$.

In particular every pure acyclic complex lies in the kernel of $H$.

\begin{defn}
With notation as above we denote by $A_{\otimes}(\Inj R)$ the quotient \\ $\ker(H)/\K_{\mathrm{pac}}(\Flat R)$, where
\begin{displaymath}
\ker(H) = \{F\in \K(\Flat R) \; \vert \; H(\Sigma^i F) = 0 \; \forall i \in \ZZ\}.
\end{displaymath}
\end{defn}

The next lemma shows this is the desired subcategory of $\sfN(\Flat R)$.

\begin{lem}[\cite{Stevensonclass}*{Lemma~4.3}]
An object $F$ of $\K(\Flat R)$ lies in $\ker(H)$ if and only if the exact functor
\begin{displaymath}
F\otimes_{R} (-) \colon \K(\Inj R) \to \K(\Inj R)
\end{displaymath}
restricts to
\begin{displaymath}
F\otimes_{R} (-) \colon \sfS(R) \to \sfS(R).
\end{displaymath}
In particular, $A_\otimes(\Inj X)$ consists of the pure quasi-isomorphism classes of objects which act on $\sfS(X)$.
\end{lem}

Before continuing we recall the notion of K-flatness. By taking K-flat resolutions we get an action of $\D(R)$ via the above category.

\begin{defn}\label{defn_Kflat}
We say that a complex of flat $R$-modules $F$ is \emph{K-flat} provided $F \otimes_{R}(-)$ sends quasi-isomorphisms to quasi-isomorphisms (or equivalently if $F\otimes_{R} E$ is an exact complex for any exact complex of $R$-modules $E$).
\end{defn}

\begin{lem}
There is a fully faithful, exact, coproduct preserving functor 
\begin{displaymath}
\D(R) \to A_{\otimes}(\Inj R).
\end{displaymath}
\end{lem}
\begin{proof}
There is, by the proof of Theorem 5.5 of \cite{MurfetThesis}, a fully faithful, exact, coproduct preserving functor $\D(R)\to \sfN(\Flat R)$ given by taking K-flat resolutions and inducing an equivalence
\begin{displaymath}
\D(R) \iso {}^{\perp}\sfN_{\mathrm{ac}}(\Flat R).
\end{displaymath}
This functor given by taking resolutions factors via $A_\otimes(\Inj R)$ since K-flat complexes send acyclics to acyclics under the tensor product.
\end{proof}

Taking K-flat resolutions and then tensoring gives the desired action
\begin{displaymath}
(-) \odot (-) \colon \D(R) \times \sfS(R) \to \sfS(R)
\end{displaymath}
 by an easy argument: K-flat resolutions are well behaved with respect to the tensor product so the necessary compatibilities follow from those of the tensor product of complexes.

\begin{rem}\label{rem_boundedabove}
Recall that every complex in $\K^{-}(\Flat R)$, the homotopy category of bounded above complexes of flat $R$-modules, is K-flat. 
 Thus when acting by the subcategory $\K^{-}(\Flat R)$ there is an equality $\odot = \otimes_{R}$.
\end{rem}

We can now apply all of the machinery we have introduced for actions of rigidly-compactly generated triangulated categories. Recall from Example~\ref{ex_ringspec} that 
\begin{displaymath}
\Spc \D(R)^c = \Spc \D^{\mathrm{perf}}(R) \iso \Spec R
\end{displaymath}
Henceforth we will identify these spaces. As in Definition~\ref{defn_actionsupport} we get a notion of support on $\sfS(R)$ with values in $\Spec R$; we will denote the support of an object $A$ of $\sfS(R)$ simply by $\supp A$. As the category $\D(R)$ has a model the local-to-global principle (Theorem \ref{thm_general_ltg}) holds. In particular, not only do we have assignments as in Definition \ref{defn_vis_sigmatau} 
\begin{equation}\label{eqithump}
\left\{ \begin{array}{c}
\text{subsets of}\; \Spec R 
\end{array} \right\}
\xymatrix{ \ar[r]<1ex>^\t \ar@{<-}[r]<-1ex>_\s &} \left\{
\begin{array}{c}
\text{localising subcategories} \; \text{of} \; \sfS(R) \\
\end{array} \right\} 
\end{equation}
where for a localising subcategory $\sfL$ we set
\begin{displaymath}
\s(\sfL) = \supp \sfL = \{\mathfrak{p} \in \Spec R \; \vert \; \mathit{\Gamma}_\mathfrak{p}\sfL \neq 0\}
\end{displaymath}
and for a subset $W\cie \Spec R$ we set
\begin{displaymath}
\t(W) = \{A \in \sfS(R) \; \vert \; \supp A \cie W\}.
\end{displaymath}
we have the following additional information.

\begin{prop}\label{prop_imtau}
Given a subset $W\cie \Spec R$ there is an equality of subcategories
\begin{displaymath}
\t(W) = \loc( \mathit{\Gamma}_\mathfrak{p}\sfS(R) \; \vert \; \mathfrak{p}\in W)
\end{displaymath}
\end{prop}
\begin{proof}
This is just a restatement of Lemma~\ref{general_im_tau} combined with Lemma~\ref{lem_big_unit_gen}.
\end{proof}

We now describe $\s\sfS(R)$, the support of the singularity category. In order to describe it we need to recall the following definition.

\begin{defn}
The \emph{singular locus} of $R$ is the set
\begin{displaymath}
\Sing(R) = \{\mathfrak{p}\in \Spec R\; \vert \; R_\mathfrak{p} \; \text{is not regular}\},
\end{displaymath}
i.e.\ the collection of those prime ideals $\mathfrak{p}$ such that $R_\mathfrak{p}$ has infinite global dimension. Recall that this is equivalent to the residue field $k(\mathfrak{p})$ having infinite projective dimension over $R_\mathfrak{p}$.
\end{defn}

\begin{prop}[\cite{Stevensonclass}*{Proposition~5.7}]\label{prop_nontriv}
For any $\mathfrak{p}\in \Sing R$ the object $\mathit{\Gamma}_\mathfrak{p}I_\l Q_\r k(\mathfrak{p})$ is non-zero in $\sfS(R)$. Thus $\mathit{\Gamma}_\mathfrak{p}\sfS(R)$ is non-trivial for all such $\mathfrak{p}$ yielding the equality
\begin{displaymath}
\s \sfS(R) = \Sing R.
\end{displaymath}
\end{prop}

We now restrict ourselves to hypersurfaces, which we define below, so we can state the classification theorem, and provide a rough sketch of the proof.

\begin{defn}\label{defn_CIring}
Let $(R,\mathfrak{m},k)$ be a noetherian local ring. We say $R$ is a \emph{hypersurface} if its $\mathfrak{m}$-adic completion $\hat{R}$ can be written as the quotient of a regular ring by a regular element (i.e.\ a non-zero divisor). 

Given a not necessarily local commutative noetherian ring $R$ if, when localized at each prime ideal $\mathfrak{p}$ in $\Spec R$, $R_\mathfrak{p}$ is a hypersurface we say that $R$ is \emph{locally a hypersurface}.
\end{defn}

\begin{thm}[\cite{Stevensonclass}*{Theorem~6.13}]\label{cor_hyper_min}
If $R$ is a noetherian ring which is locally a hypersurface then there is an isomorphism of lattices
\begin{displaymath}
\left\{ \begin{array}{c}
\text{subsets of}\; \Sing R 
\end{array} \right\}
\xymatrix{ \ar[r]<1ex>^{\tau} \ar@{<-}[r]<-1ex>_{\sigma} &} \left\{
\begin{array}{c}
\text{localising subcategories of} \; \sfS(R) \\
\end{array} \right\} 
\end{displaymath}
given by the assignments of (\ref{eqithump}). This restricts to the equivalent lattice isomorphisms
\begin{displaymath}
\left\{ \begin{array}{c}
\text{specialization closed} \\ \text{subsets of}\; \Sing R 
\end{array} \right\}
\xymatrix{ \ar[r]<1ex>^{\tau} \ar@{<-}[r]<-1ex>_{\sigma} &} \left\{
\begin{array}{c}
\text{localising subcategories of} \; \sfS(R)\; \\
\text{generated by objects of} \; \sfS(R)^c
\end{array} \right\} 
\end{displaymath}
and
\begin{displaymath}
\left\{ \begin{array}{c}
\text{specialization closed} \\ \text{subsets of}\; \Sing R 
\end{array} \right\}
\xymatrix{ \ar[r]<1ex> \ar@{<-}[r]<-1ex> &} \left\{
\begin{array}{c}
\text{thick subcategories of} \; \D_{\mathrm{Sg}}(R) \\
\end{array} \right\}.
\end{displaymath}
\end{thm}
\begin{proof}[Sketch of proof]
We only sketch the proof of the first bijection, the others follow from it but the proof is somewhat more involved.

By Proposition~\ref{general_tau_inj}, combined with the computation of $\s\sfS(R)$ given in Proposition~\ref{prop_nontriv}, we know that $\t$ is injective with left inverse $\s$. Suppose then that $\sfL$ is a localising subcategory and consider
\begin{displaymath}
\t\s\sfL = \loc(\mathit{\Gamma}_\mathfrak{p}\sfS(R) \; \vert \; \mathit{\Gamma}_\mathfrak{p}\sfL\neq 0)
\end{displaymath}
where we have this equality by Proposition~\ref{prop_imtau}. By \cite{Stevensonclass}*{Theorem~6.12} (together with a short reduction to the local case) each of the localising subcategories $\mathit{\Gamma}_\mathfrak{p}\sfS(R)$ is \emph{minimal}, i.e.\ it has no proper non-zero localising subcategories. Thus, if $\mathit{\Gamma}_\mathfrak{p}\sfL$ is non-zero it must, by this minimality, be equal to $\mathit{\Gamma}_\mathfrak{p}\sfS(R)$. By the local-to-global principle we know
\begin{displaymath}
\sfL = \loc(\mathit{\Gamma}_\mathfrak{p}\sfL \; \vert \; \mathfrak{p}\in \Sing R)
\end{displaymath}
and so putting all this together yields
\begin{displaymath}
\sfL = \loc(\mathit{\Gamma}_\mathfrak{p}\sfL \; \vert \; \mathfrak{p}\in \Sing R) = \loc(\mathit{\Gamma}_\mathfrak{p}\sfS(R) \; \vert \; \mathit{\Gamma}_\mathfrak{p}\sfL\neq 0) = \t\s\sfL.
\end{displaymath}
This proves we have the claimed bijection.
\end{proof}

\begin{rem}
One should compare the above approach with that used in Section~\ref{sec_Neeman_ex} (and note the similarity). The strategy evident in these proof sketches is really the thing to remember. Given the local-to-global principle for an action of $\sfT$ on $\K$ it is sufficient to check that each $\mathit{\Gamma}_x\K$ is either zero or has no non-trivial, proper localising submodules; if this is the case one has an isomorphism of lattices
\begin{displaymath}
\Loc^*(\K) \cong \{\text{subsets of} \; \s\K\}
\end{displaymath}
giving a solution to the classification problem for submodules.
\end{rem}

\subsection{Representations of categories over commutative noetherian rings}

As another application of the local-to-global principle we sketch here some ideas coming from recent work of the author with Benjamin Antieau \cite{StevensonAntieau}. We are quite light on details; the idea is just to give another taste of the kind of reductions one can do with the local-to-global principle. The idea is to reduce the study of localising subcategories of derived categories of representations of a small category in $R$-modules, for a commutative noetherian ring $R$, to the study of the same problem over fields. This allows one, for instance, to give a classification of localising subcategories in the derived category of the $R$-linear path algebra of a Dynkin quiver.

We begin with some setup. Again we do not work in the maximum possible generality, but restrict the discussion to make it less technical. Fix a commutative noetherian ring $R$ and let $\sfC$ be a small category. For instance $\sfC$ could be the path category of a quiver, a groupoid, or a poset. The category of (right) $R$-modules is the category
\begin{displaymath}
\Modu_R \sfC = [\sfC^\mathrm{op}, \Modu R]
\end{displaymath}
of contravariant functors from $\sfC$ to $R$-modules. By standard nonsense $\Modu_R \sfC$ is a Grothendieck abelian category with enough projectives and is equivalent to the category of $R$-linear functors
\begin{displaymath}
[R\sfC^\mathrm{op}, \Modu R]_R
\end{displaymath}
where $R\sfC$ denotes the free $R$-linear category on $\sfC$. There is an action
\begin{displaymath}
\Modu R \times \Modu_R \sfC \stackrel{\otimes_R}\to \Modu_R \sfC
\end{displaymath}
defined by pointwise tensoring, i.e.\ given an $R$-module $M$ and a $\sfC$-module $F$ we have for $c\in \sfC$
\begin{displaymath}
(M\otimes_R F)(c) = M\otimes_R(F(c)).
\end{displaymath}

Given a homomorphism of rings $\phi\colon R\to S$ we have the obvious base change and restriction functors giving rise to an adjunction
\begin{displaymath}
\xymatrix{
\Modu_R \sfC \ar[r]<0.5ex>^-{\phi^*} \ar@{<-}[r]<-0.5ex>_-{\phi_*} & \Modu_S \sfC
}
\end{displaymath}
and we can write $\phi^*$ in terms of the action as $S\otimes_R -$ together with the obvious $S$-action.

We can now derive all of this to obtain an action
\begin{displaymath}
\D(R) \times \D(R\sfC) \to \D(R\sfC)
\end{displaymath}
where $\D(R\sfC)$ is our shorthand for the unbounded derived category $\D(\Modu_R \sfC)$ of $\sfC$-modules. As the projectives in $\Modu_R \sfC$ are pointwise $R$-flat we can resolve in $\D(R\sfC)$ to compute this left derived action and so given $\phi\colon R\to S$ as above $S\otimes^\mathbf{L}_R-$ realises $\mathbf{L}\phi^*$. As a particular instance of this we can take, for $\mathfrak{p}\in \Spec R$, the morphism $R\to k(\mathfrak{p})$. We thus get an adjunction
\begin{displaymath}
\xymatrix{
\D(R\sfC) \ar[rr]<0.5ex>^-{k(\mathfrak{p})\otimes^\mathbf{L}_R -} \ar@{<-}[rr]<-0.5ex> && \D(k(\mathfrak{p})\sfC)
}
\end{displaymath}
relating representations over $R$ to representations over $k(\mathfrak{p})$ and arising from the $\D(R)$ action.

We can use these adjunctions to split up the problem of understanding the lattice of localising subcategories of $\D(R\sfC)$ into pieces. Define
\begin{displaymath}
\mathcal{F} \stackrel{s}{\to} \Spec R
\end{displaymath}
by declaring that $s^{-1}(\mathfrak{p}) = \Loc(\D(k(\mathfrak{p})\sfC))$, i.e.\ we have
\begin{displaymath}
\mathcal{F} = \coprod_{\mathfrak{p}\in \Spec R} \Loc(\D(k(\mathfrak{p})\sfC))
\end{displaymath}
and $s$ is the obvious ``collapsing map''. We get poset morphisms
\begin{displaymath}
\Loc\D(R\sfC)
\xymatrix{ \ar[r]<1ex>^f \ar@{<-}[r]<-1ex>_g &} 
\left\{ \begin{array}{c}
\text{sections}\; l\; \text{of} \\ \mathcal{F}\stackrel{s}{\to}\Spec R
\end{array} \right\}
\end{displaymath}
defined on a localising subcategory $\sfL$ and a section $l$ by
\begin{displaymath}
f(\sfL) \colon \mathfrak{p} \mapsto \loc(k(\mathfrak{p})\otimes \sfL) \subseteq \D(k(\mathfrak{p})\sfC)
\end{displaymath}
and
\begin{displaymath}
g(l) = \{X\in \D(R\sfC) \; \vert \; k(\mathfrak{p})\otimes X \in l(\mathfrak{p}) \; \forall\; \mathfrak{p}\in \Spec R \}.
\end{displaymath}

One of the main theorems of \cite{StevensonAntieau} is: 
\begin{thm}
The morphisms $f$ and $g$ are inverse to one another.
\end{thm}
This fact is a direct application of the local-to-global principle for the action of $\D(R)$ on $\D(R\sfC)$. Let us close by giving an outline of the proof.

By Theorem~\ref{thm_general_ltg} applied to the $\D(R)$ action we can reduce to checking that there is a lattice isomorphism
\begin{displaymath}
\Loc \mathit{\Gamma}_\mathfrak{p}\D(R\sfC) \cong \Loc \D(k(\mathfrak{p})\sfC)
\end{displaymath}
for each $\mathfrak{p}\in \Spec R$. However, one can proceed more directly. The main inputs for this direct approach to proving the theorem are the following two technical results. The first is essentially a consequence of the observation we have made in Lemma~\ref{lem_kpyo}

\begin{prop}
Given $X\in \D(R\sfC)$ there is an equality
\begin{displaymath}
\loc(X) = \loc(k(\mathfrak{p})\otimes X\; \vert \; \mathfrak{p}\in \Spec R).
\end{displaymath}
In particular, $k(\mathfrak{p})\otimes X\cong 0$ for all $\mathfrak{p}$ if and only if $X\cong 0$. 
\end{prop}

\begin{prop}
Given a section $l$ of $\mathcal{F}\stackrel{s}{\to}\Spec R$ there is an equality
\begin{displaymath}
g(l) = \{X\; \vert \; k(\mathfrak{p})\otimes X\in l(\mathfrak{p}) \; \forall\; \mathfrak{p}\in \Spec R\} = \loc(l(\mathfrak{p})\; \vert \; \mathfrak{p}\in \Spec R)
\end{displaymath}
where $l(\mathfrak{p})$ is viewed as a subcategory of $\D(R\sfC)$ via restricting along $R\to k(\mathfrak{p})$.
\end{prop}

Fix a localising subcategory $\sfL$ of $\D(R\sfC)$ and a section $l$ of $s$. Given these two results one proves the theorem via the computations
\begin{align*}
gf(\sfL) &=\loc(f(\sfL)(\mathfrak{p})\; \vert \; \mathfrak{p}\in \Spec R) \\
&= \loc(k(\mathfrak{p}) \otimes \sfL\; \vert \; \mathfrak{p}\in \Spec R) \\
&= \sfL
\end{align*}
where the final equality holds by the first proposition, and
\begin{align*}
fg(l)(\mathfrak{p}) &= \loc(k(\mathfrak{p})\otimes g(l))\\
&= \loc(k(\mathfrak{p})\otimes l(\mathfrak{p})) \\
&= l(\mathfrak{p}).
\end{align*}

It remains an open and challenging problem to better understand the smashing subcategories and telescope conjecture for $\D(R\sfC)$. Furthermore, very little is known about the lattices $\Loc \D(k\sfC)$ for a field $k$ except in certain very special cases.

\bibliography{greg_bib}

\begin{bibdiv}
\begin{biblist}

\bib{StevensonAntieau}{article}{
      author={Antieau, Benjamin},
      author={Stevenson, Greg},
       title={Derived categories of representations of small categories over
  commutative noetherian rings},
        date={2015},
     journal={Pacific J. Math},
        note={to appear},
}

\bib{BaSpec}{article}{
      author={Balmer, Paul},
       title={The spectrum of prime ideals in tensor triangulated categories},
        date={2005},
     journal={J. Reine Angew. Math.},
      volume={588},
       pages={149\ndash 168},
}

\bib{BaFilt}{article}{
      author={Balmer, Paul},
       title={Supports and filtrations in algebraic geometry and modular
  representation theory},
        date={2007},
     journal={Amer. J. Math.},
      volume={129},
      number={5},
       pages={1227\ndash 1250},
}

\bib{BaSSS}{article}{
      author={Balmer, Paul},
       title={Spectra, spectra, spectra - {T}ensor triangular spectra versus
  {Z}ariski spectra of endomorphism rings},
        date={2010},
     journal={Algebraic and Geometric Topology},
      volume={10},
      number={3},
       pages={1521\ndash 1563},
}

\bib{BCR}{article}{
      author={Benson, Dave},
      author={Carlson, Jon~F.},
      author={Rickard, Jeremy},
       title={Thick subcategories of the stable module category},
        date={1997},
     journal={Fund. Math.},
      volume={153},
      number={1},
       pages={59\ndash 80},
}

\bib{BaRickard}{article}{
      author={Balmer, Paul},
      author={Favi, Giordano},
       title={Generalized tensor idempotents and the telescope conjecture},
        date={2011},
     journal={Proc. Lond. Math. Soc. (3)},
      volume={102},
      number={6},
       pages={1161\ndash 1185},
}

\bib{BIKstrat2}{article}{
      author={Benson, Dave},
      author={Iyengar, Srikanth~B.},
      author={Krause, Henning},
       title={Stratifying triangulated categories},
        date={2011},
     journal={J. Topol.},
      volume={4},
      number={3},
       pages={641\ndash 666},
}

\bib{BKS}{article}{
      author={Buan, Aslak~Bakke},
      author={Krause, Henning},
      author={Solberg, {\O}yvind},
       title={Support varieties: an ideal approach},
        date={2007},
     journal={Homology, Homotopy Appl.},
      volume={9},
      number={1},
       pages={45\ndash 74},
}

\bib{NeeHolim}{article}{
      author={B{\"o}kstedt, Marcel},
      author={Neeman, Amnon},
       title={Homotopy limits in triangulated categories},
        date={1993},
     journal={Compositio Math.},
      volume={86},
      number={2},
       pages={209\ndash 234},
}

\bib{BousfieldBAS}{article}{
      author={Bousfield, A.~K.},
       title={Correction to: ``{T}he {B}oolean algebra of spectra'' [{C}omment.
  {M}ath. {H}elv. {\bf 54} (1979), no. 3, 368--377; ]},
        date={1983},
     journal={Comment. Math. Helv.},
      volume={58},
      number={4},
       pages={599\ndash 600},
}

\bib{Buchweitzunpub}{article}{
      author={Buchweitz, Ragnar-Olaf},
       title={Maximal {C}ohen-{M}acaulay modules and {T}ate cohomology over
  {G}orenstein rings},
        date={1987},
     journal={unpublished},
}

\bib{DevinatzHopkinsSmith}{article}{
      author={Devinatz, Ethan~S.},
      author={Hopkins, Michael~J.},
      author={Smith, Jeffrey~H.},
       title={Nilpotence and stable homotopy theory. {I}},
        date={1988},
     journal={Ann. of Math. (2)},
      volume={128},
      number={2},
       pages={207\ndash 241},
}

\bib{GreenleesTate}{incollection}{
      author={Greenlees, J. P.~C.},
       title={Tate cohomology in axiomatic stable homotopy theory},
        date={2001},
   booktitle={Cohomological methods in homotopy theory ({B}ellaterra, 1998)},
      series={Progr. Math.},
      volume={196},
   publisher={Birkh\"auser},
     address={Basel},
       pages={149\ndash 176},
}

\bib{GrothPSD}{article}{
      author={Groth, Moritz},
       title={Derivators, pointed derivators, and stable derivators},
        date={2013},
     journal={Algebr. Geom. Topol.},
      volume={13},
       pages={313\ndash 374},
}

\bib{Happeltricat}{book}{
      author={Happel, Dieter},
       title={Triangulated categories in the representation theory of
  finite-dimensional algebras},
      series={London Mathematical Society Lecture Note Series},
   publisher={Cambridge University Press},
     address={Cambridge},
        date={1988},
      volume={119},
}

\bib{HartshorneLC}{book}{
      author={Hartshorne, Robin},
       title={Local cohomology},
      series={A seminar given by A. Grothendieck, Harvard University, Fall},
   publisher={Springer-Verlag},
     address={Berlin},
        date={1967},
      volume={1961},
}

\bib{HochsterSpectral}{article}{
      author={Hochster, M.},
       title={Prime ideal structure in commutative rings},
        date={1969},
     journal={Trans. Amer. Math. Soc.},
      volume={142},
       pages={43\ndash 60},
}

\bib{HPS}{article}{
      author={Hovey, Mark},
      author={Palmieri, John~H.},
      author={Strickland, Neil~P.},
       title={Axiomatic stable homotopy theory},
        date={1997},
     journal={Mem. Amer. Math. Soc.},
      volume={128},
      number={610},
       pages={x+114},
}

\bib{KockPitsch}{article}{
      author={Kock, Joachim},
      author={Pitsch, Wolfgang},
       title={Hochster duality in derived categories and point-free
  reconstruction of schemes},
        date={2013},
     journal={arXiv:1305.1503},
}

\bib{KrStab}{article}{
      author={Krause, Henning},
       title={The stable derived category of a {N}oetherian scheme},
        date={2005},
     journal={Compos. Math.},
      volume={141},
      number={5},
       pages={1128\ndash 1162},
}

\bib{KrLoc}{incollection}{
      author={Krause, Henning},
       title={Localization theory for triangulated categories},
        date={2010},
   booktitle={Triangulated categories},
      series={London Math. Soc. Lecture Note Ser.},
      volume={375},
   publisher={Cambridge Univ. Press, Cambridge},
       pages={161\ndash 235},
}

\bib{MurfetTAC}{article}{
      author={Murfet, Daniel},
      author={Salarian, Shokrollah},
       title={Totally acyclic complexes over {N}oetherian schemes},
        date={2011},
     journal={Adv. Math.},
      volume={226},
      number={2},
       pages={1096\ndash 1133},
}

\bib{MurfetThesis}{article}{
      author={Murfet, Daniel},
       title={The mock homotopy category of projectives and {G}rothendieck
  duality},
        date={2007},
     journal={Ph.D. thesis},
        note={Available at: \url{http://www.therisingsea.org/thesis.pdf}},
}

\bib{NeeCat}{book}{
      author={Neeman, Amnon},
       title={Triangulated categories},
      series={Annals of Mathematics Studies},
   publisher={Princeton University Press},
     address={Princeton, NJ},
        date={2001},
      volume={148},
        ISBN={0-691-08685-0; 0-691-08686-9},
      review={\MR{MR1812507 (2001k:18010)}},
}

\bib{NeeFlat}{article}{
      author={Neeman, Amnon},
       title={The homotopy category of flat modules, and {G}rothendieck
  duality},
        date={2008},
     journal={Invent. Math.},
      volume={174},
      number={2},
       pages={255\ndash 308},
}

\bib{NeeChro}{article}{
      author={Neeman, Amnon},
       title={The chromatic tower for {$D(R)$}},
        date={1992},
     journal={Topology},
      volume={31},
      number={3},
       pages={519\ndash 532},
        note={With an appendix by Marcel B{\"o}kstedt},
}

\bib{NeeGrot}{article}{
      author={Neeman, Amnon},
       title={The {G}rothendieck duality theorem via {B}ousfield's techniques
  and {B}rown representability},
        date={1996},
     journal={J. Amer. Math. Soc.},
      volume={9},
      number={1},
       pages={205\ndash 236},
}

\bib{Rickardidempotent}{article}{
      author={Rickard, Jeremy},
       title={Idempotent modules in the stable category},
        date={1997},
     journal={J. London Math. Soc. (2)},
      volume={56},
      number={1},
       pages={149\ndash 170},
}

\bib{StevensonActions}{article}{
      author={Stevenson, Greg},
       title={Support theory via actions of tensor triangulated categories},
        date={2013},
     journal={J. Reine Angew. Math.},
      volume={681},
       pages={219\ndash 254},
}

\bib{StevensonAbsFlat}{article}{
      author={Stevenson, Greg},
       title={Derived categories of absolutely flat rings},
        date={2014},
     journal={Homology Homotopy Appl.},
      volume={16},
      number={2},
       pages={45\ndash 64},
}

\bib{StevensonLTG}{article}{
      author={Stevenson, Greg},
       title={The local-to-global principle for triangulated categories via
  dimension functions},
        date={2016},
     journal={arXiv:1601.01205},
}

\bib{Stevensonclass}{article}{
      author={Stevenson, Greg},
       title={Subcategories of singularity categories via tensor actions},
        date={2014},
     journal={Compos. Math.},
      volume={150},
      number={2},
       pages={229\ndash 272},
}

\bib{Thomclass}{article}{
      author={Thomason, R.~W.},
       title={The classification of triangulated subcategories},
        date={1997},
     journal={Compositio Math.},
      volume={105},
      number={1},
       pages={1\ndash 27},
}

\end{biblist}
\end{bibdiv}

\end{document}